\documentclass{amsart}
\usepackage{mathtools}
\usepackage{comment}
\usepackage{amssymb}
\usepackage{amsmath}
\usepackage{amsthm}
\usepackage{yhmath}
\usepackage{enumitem}
\usepackage{xcolor}
\usepackage{hyperref}
\hypersetup{
           breaklinks=true,   
           colorlinks=true,   
           pdfusetitle=true,  
        }
\numberwithin{equation}{section}
\theoremstyle{plain}
\newtheorem{theorem}[equation]{Theorem}
\newtheorem{proposition}[equation]{Proposition}
\newtheorem{lemma}[equation]{Lemma}
\newtheorem{corollary}[equation]{Corollary}
\newtheorem*{claim*}{Claim}
\theoremstyle{definition}
\newtheorem{definition}[equation]{Definition}
\newtheorem{setup}[equation]{Setup}
\theoremstyle{remark}
\newtheorem{remark}[equation]{Remark}

\author[R.\ Murakami]{Rei Murakami}
\address{Mathematical Institute, Tohoku University, 6-3, Aramaki Aza-Aoba, Aoba-ku, Sendai 980-8578, Japan}
\email{rei.murakami.p3@dc.tohoku.ac.jp, reimurakami66@gmail.com}

\begin{document}

\title[$J$-EQUATIONS AND DHYM EQUATIONS ON HOLOMORPHIC SUBMERSIONS]{$J$-EQUATIONS AND DEFORMED HERMITIAN-YANG-MILLS EQUATIONS ON HOLOMORPHIC SUBMERSIONS}

\begin{abstract} 
In this paper, we prove that there exists a solution of the $J$-equation on the total space of a holomorphic submersion if there exist solutions of the $J$-equation on the fibers and the base.
The method is an adiabatic limit technique. 
We also partially prove the converse implication. More precisely, if the total space is $J$-nef, then each fiber is $J$-nef. In addition, if each fiber has a solution of the $J$-equation, then the base is also $J$-nef. 
Furthermore, we establish similar phenomena for the deformed Hermitian-Yang-Mills equation.
\end{abstract}

\maketitle

\section{Introduction}\label{sec:Introduction}
 
 Finding a canonical metric such as a constant scalar curvature K{\"a}hler (cscK, for short) metric in a given K{\"a}hler class is a central problem in K{\"a}hler geometry. 
This paper studies when the total space of a holomorphic submersion admits a canonical metric. Fine \cite{Fine} constructed cscK metrics on fibered complex surfaces admitting holomorphic submersions onto high genus curves with fibers of genus at least two.
Dervan-Sektnan \cite{Dervan-Sektnan} generalized the above result to the higher dimensional case, introducing a new concept of relatively K{\"a}hler metrics on fibrations called \textit{optimal symplectic connections}.
 
In this paper, instead of cscK metrics, we consider solutions of the $J$-equations on holomorphic submersions. 
The $J$-equation is related to cscK metrics. For example, if $\mathrm{Ric}(\omega)<0$ and there exists a solution of the $J_{-\mathrm{Ric}(\omega)}$-equation in a K{\"a}hler class $[\omega]$, then there exists a cscK metric in $[\omega]$, by the properness-solvability equivalence of the cscK equation \cite{CC} and the $J$-equation \cite[Propositions 21 and 22]{CS}.
For a compact K{\"a}hler manifold $(X,\chi)$ of dimension $n$, the $J_{\chi}$-equation is given by 
\begin{equation}\label{eq:J-eq}
\Lambda_\omega\chi = c,\
\end{equation}
where $\Lambda_\omega$ denotes the trace with respect to a K{\"a}hler form $\omega$ and $c$ is the constant determined by 
$${n \int_X \chi \wedge \omega^{n-1}} = c\int_X \omega^n.$$
Let us fix our setup.
\begin{setup}\label{setup}
Let $(X,\chi)$ and $B$ be compact K{\"a}hler manifolds and $\pi : X \to B$ be a holomorphic submersion of relative dimension $m$, and denote the dimension of $B$ by $n$. 
Suppose that $\omega_X$ is a relatively K{\"a}hler form on $X$, i.e. the restriction on each fiber is K{\"a}hler, and $\omega_B$ is a K{\"a}hler form on $B$.  
\end{setup}

In this setup, the tangent space $TX$ splits as a smooth bundle 
$$TX \cong \mathcal{V} \oplus \mathcal{H}, $$ 
where $\mathcal{V} = \mathop{\mathrm{ker}}d\pi $ denotes the vertical tangent bundle and $\mathcal{H}$ denotes the horizontal subbundle of $TX$ defined by 
$$ \mathcal{H}_x= \{ u \in T_x X \mid \omega_X (u, v)=0 \; \textrm{for all} \; v \in \mathcal{V}_x \}.$$ 
By this splitting, the K{\"a}hler form $\chi$ on $X$ is divided into the purely vertical component $\chi_\mathcal{V}$, the purely horizontal component $\chi_\mathcal{H}$ and the mixed component $\chi_m$. 
Denote by $\chi_b$ the restriction of $\chi$ to a fiber $X_b$.
Define a $(1,1)$-form $\pi_B(\chi_\mathcal{H})$ on $B$ by fiberwise integral of $\chi_\mathcal{H} \wedge \omega_X^m $ divided by the volume of a fiber $X_b$, i.e.,
$$ \pi_B(\chi_\mathcal{H})(b) =  V_b^{-1} \pi_*(\chi_\mathcal{H} \wedge \omega_X^m)(b),$$
where $V_b = \int_{X_b} \omega_b^m $ and $\omega_b$ is the restriction of $\omega_X$ to a fiber $X_b$. See Subsection \ref{sec:pre} for the precise definition. Note that $V_b$ is independent of $b$ since $\omega_X$ is closed. Define a constant $c_b$ by 
$$c_b = \frac{m \int_{X_b} {\chi_b} \wedge \left( \omega_b \right)^{m-1}}{\int_{X_b} \left( \omega_b \right)^m}.$$
Note that $c_b$ is independent of $b$ since $\chi$ and $\omega_X$ are closed.
The main theorems of this paper are analogues of the following result by Dervan-Sektnan. 

\begin{theorem}[{\cite[Theorem 1.2]{Dervan-Sektnan}}]\label{thm:DS}
Suppose that $(B, L)$ admits a twisted cscK metric $\omega_B \in c_1(L)$ and $\pi: (X,H) \rightarrow (B,L)$ admits an optimal symplectic connection $\omega_X \in c_1(H)$. Assume that both of the automorphism groups $\mathrm{Aut}(X,H)$ and $\mathrm{Aut}(q)$, where $q$ is the moduli map, are discrete. Then there exists a cscK metric in the class $k\pi^*c_1(L) + c_1(H)$ for $k \gg 0$.
\end{theorem}
A relatively K{\"a}hler form $\omega_X$ whose restriction to each fiber is a cscK metric is called an \textit{optimal symplectic connection} if it satisfies a certain equation. If a cscK metric on each fiber $(X_b, H_b)$ is unique, then the optimal symplectic connection condition becomes trivial. Our situation is similar to this situation, since a solution of the $J$-equation is unique \cite[Proposition 2]{Chen}.
We suppress pullbacks via $\pi$, so if $\omega_B$ is a form on $B$, its pullback to $X$ will also be denoted by $\omega_B$.

\begin{theorem}\label{thm:main}
In Setup \ref{setup}, assume that 
the restriction $\omega_b$ to each fiber $X_b$ is a solution of the $J_{\chi_b}$-equation and   
$\omega_B$ is a solution of the $J_{\pi_B(\chi_\mathcal{H})}$-equation.
Then there exists a solution of the $J_\chi$-equation in the class $[ {\omega_X} + k \omega_B]$ for $k \gg 0$.
\end{theorem}
We provide two proofs of this theorem. One uses the $\mathcal{C}$-subsolution criterion of the existence of a solution of the $J$-equation due to \cite{SW} (Theorem \ref{thm:SW}). The other uses an adiabatic limit technique inspired by \cite{Fine, DS20, Dervan-Sektnan}.
In particular, in \cite[Theorem 1.4]{DS20}, they considered the case where a twisting form on the total space is the pullback of a K{\"a}hler form on the base.  
The difference from those earlier results is that we choose an arbitrary reference metric $\chi$ on the total space $X$.

We also consider the converse implication of Theorem \ref{thm:main}. 
We use a topological condition called \textit{$J$-positivity}. 
The Lejmi-Sz{\'e}kelyhidi conjecture \cite{LS} says that the $J_\chi$-equation \eqref{eq:J-eq} is solvable in a K{\"a}hler class $[\omega]$ on a compact K{\"a}hler manifold $X$ of dimension $n$ if and only if we have 
$$ \int_W{ c \, {\omega}^p - p \, {\omega}^{p-1} \wedge \chi} > 0$$
for all $p$-dimensional subvarieties $W \subset X$, where $ p = 1,2, \dots ,n-1$.
A pair $([\omega], [\chi])$ is said to be $J$-positive if the latter condition holds.
If it is just nonnegative, a pair is said to be \textit{$J$-nef}.
The uniform version of the conjecture is proved by Gao Chen \cite{G.Chen} and it is also proved that the uniform conditions are equivalent to the uniform $J$-stability. 
The original version of the conjecture is proved by Datar-Pingali \cite{Datar-Pingali} in the projective case and by Song \cite{Song} in general.

Now we state the converse implication.

\begin{theorem}\label{thm:converse}
In Setup \ref{setup}, if the pair $([\omega_X + k \omega_B], [\chi])$ is $J$-nef for $ k \gg 0$, then the pair $([\omega_b], [\chi_b])$ is $J$-nef for all $ b \in B$. 
In addition, if the restriction $\omega_b$ is a solution of the $J_{\chi_b}$-equation or $\pi^*(\pi_B(\omega_X)_\mathcal{H}) = (\omega_X)_\mathcal{H} $, then the pair $([\omega_B], [\pi_B\left(\chi_\mathcal{H}\right)])$ is also $J$-nef.
\end{theorem}

The calculation also shows that if one has a uniform upper bound of a solution of the $J$-equation on the total space, then there exists a solution of the $J$-equation on a base (see Remark \ref{rem:collapsing}). This might be related to the work \cite{GPT}, where they showed the $L^\infty$ estimate for solutions of a family of Hessian equations with a certain structural condition. 

Furthermore, we prove similar theorems for the deformed Hermitian-Yang-Mills (dHYM, for short) equation.
On an $n$-dimensional compact K{\"a}hler manifold $(X,\chi)$ with a closed real $(1,1)$-form $\omega$, the dHYM$_\chi$ equation is given by
\begin{equation}\label{eq:dhym}
    \mathrm{Im}\left(e^{-\sqrt{-1}\theta}(\omega+\sqrt{-1}\chi)^n\right)=0, \quad \mathrm{Re}\left(e^{-\sqrt{-1}\theta}(\omega+\sqrt{-1}\chi)^n\right)>0,
\end{equation}
where $\theta$ is a constant (determined up to $2\pi$) which satisfies
\begin{equation}\label{eq:dhymconst}
    \int_X \mathrm{Im}\left(e^{-\sqrt{-1}\theta} (\omega+\sqrt{-1}\chi)^n\right)=0, \quad \int_X\mathrm{Re}\left(e^{-\sqrt{-1}\theta}(\omega+\sqrt{-1}\chi)^n\right)>0.
\end{equation}
The dHYM equation was introduced by \cite{MMMS} in the physical literature and by \cite{LYZ} in the mathematical literature. The latter studied the mirror symmetry in the context of \cite{SYZ}, and derived the dHYM equation as an equation that the mirror of a special Lagrangian submanifold must satisfy. 

Let us fix our setup for the dHYM version of Theorem \ref{thm:main}.

\begin{setup}\label{setup:dhym}
    Let $(X,\chi)$ and $B$ be compact K{\"a}hler manifolds and $\pi:X\to B$ a holomorphic submersion of relative dimension $m$. Denote the dimension of $B$ by $n$. Let $\omega_X$ be a closed real $(1,1)$-form on $X$ and $\omega_B$ a K{\"a}hler form on $B$. 
\end{setup}
From this data, we can define a $(1,1)$-form $\chi_B$, and it becomes K{\"a}hler if we assume that the restriction $\omega_b$ on each fiber $X_b$ is a solution of the dHYM$_{\chi_b}$ equation (Lemma \ref{lem:dhymkahler}). The following is an analogue of Theorem \ref{thm:main}:

\begin{theorem}\label{thm:dhym_main}
    In Setup \ref{setup:dhym}, assume that 
    the restriction $\omega_b$ to each fiber $X_b$ is a solution of the dHYM$_{\chi_b}$-equation and   
    $\omega_B$ is a solution of the $J_{\chi_B}$-equation.
    Then there exists a solution of the dHYM$_\chi$-equation in the class $[ {\omega_X} + k \omega_B]$ for $k \gg 0$.
\end{theorem}

In Theorem \ref{thm:dhym_main}, remark that we do not assume the supercritical phase condition, under which the dHYM equation has been studied extensively. 
For instance, the solvability of the supercritical dHYM equation is proved to be equivalent to the existence of a supercritical $\mathcal{C}$-subsolution by \cite{CJY} (Theorem \ref{thm:CJY}), and
to a certain topological condition by \cite{G.Chen, CLT} (Theorem \ref{thm:CLT}). 
Using the characterization by \cite{CJY}, we provide the first proof of Theorem \ref{thm:dhym_main}, which uses a supercritical $\mathcal{C}$-subsolution and thus works only under the supercritical condition. The second proof uses an adiabatic limit technique and proves Theorem \ref{thm:dhym_main} in general. The key point is that this argument essentially only requires the invertibility of the linearized operator, and for the dHYM equation the linearized operator remains invertible even outside the supercritical phase range.

Although one usually assumes that $\chi$ is positive, the $J_\chi$-equation and dHYM$_\chi$ equation still make sense without this assumption. However, when $\chi$ is only semipositive, the kernel of the linearized operator is nontrivial (it contains all functions with zero $\chi$–norm), so invertibility fails. Hence, unlike the extension beyond the supercritical phase for the dHYM equation, the method cannot be pushed beyond the usual positivity assumption on $\chi$.

Related to the characterization of the solvability of the supercritical dHYM equation by \cite{CLT}, we also prove a variant of Theorem \ref{thm:converse} for the dHYM equation (Theorem \ref{thm:dhym_converse}).

\begin{subsection}*{Outline}
The paper is organized as follows. In Section \ref{sec:J}, we study the $J$-equation.
In Subsection \ref{sec:pre}, we collect the basic materials needed to prove Theorem \ref{thm:main}. 
In Subsection \ref{sec:approximation}, we construct a family of approximate solutions of the $J_\chi$-equation. Then, in Subsection \ref{sec:firstproof}, by using the first order approximation and the $\mathcal{C}$-subsolution criterion by \cite{SW}, we prove Theorem \ref{thm:main}.
In Subsection \ref{sec:proof}, we provide the second proof of Theorem \ref{thm:main}, which captures the asymptotic behavior of the solution of the $J_\chi$-equation in $[\omega_X+k\omega_B]$ for $k\gg0$.
In the second proof, we perturb the approximate solution to a genuine solution by the inverse function theorem.
In Subsection \ref{sec:converse}, by a simple calculation, we prove Theorem \ref{thm:converse}. 
We study the dHYM equation in Section \ref{sec:dHYM}. The outline of this section is the same as that of Section \ref{sec:J}. In Section \ref{sec:examples}, we provide some examples in which our main results can be applied.
\end{subsection}

\begin{subsection}*{Acknowledgements}
The author would like to thank his advisor, Shin-ichi Matsumura, for introducing this topic, carefully reading the manuscript multiple times, giving valuable comments, and for his continuous support.  
The author also thanks Ryosuke Takahashi for his warm encouragement and helpful comments.  
The author is deeply grateful to the anonymous referee for suggesting the study of the dHYM equation and a proof using the $\mathcal{C}$-subsolution, as well as for many other helpful comments that improved the clarity and overall quality of the paper.  
The author thanks Hikaru Yamamoto, who also suggested the study of the dHYM equation.  
This work was supported by JSPS KAKENHI Grant Number JP24KJ0346.
\end{subsection}

\section{$J$-equation}\label{sec:J}
\subsection{Preliminaries}\label{sec:pre}
We collect some basic properties of the trace operator needed to prove Theorem \ref{thm:main}. 
Let $(X, \chi)$ be a compact K{\"a}hler manifold of dimension $n$ and $\omega$ be a K{\"a}hler form on $X$. 
By the $\partial \bar{\partial}$-lemma, any K{\"a}hler form in a class $[\omega]$ can be described as $\omega_\phi = \omega + \sqrt{-1} \partial \bar{\partial} \phi$, where $\phi \in C^\infty(X, \mathbb{R})$.
The $J_\chi$-equation in $[\omega]$ is given by 
$$ \Lambda_{\omega_\phi} \chi = c,$$ 
where $\Lambda_{\omega_\phi}$ denotes the trace with respect to a K{\"a}hler metric $\omega_\phi$ and $c$ is the constant determined by 
$${n \int_X \chi \wedge \omega^{n-1}} = c\int_X \omega^n.$$
Let $\omega_t = \omega + t \sqrt{-1} \partial \bar{\partial} \phi $.
The linearization of the trace operator at $\omega$ is given by 
\begin{equation}\label{eq:linearisation}
\frac{d}{dt}\Big|_{t = 0} \Lambda_{\omega_t} \chi = 
{- ( \chi , \sqrt{-1} \partial \bar{\partial} \phi)_\omega } =
{- {g^{m \bar{q}}} \, {\partial_m \partial_{\bar{n}} \phi} \, {g^{p \bar{n}}} \, {\chi_{p \bar{q}}}},
\end{equation}
where $g$ is the metric tensor corresponding to $\omega$ and $(\cdot,\cdot)_\omega $ denotes the inner product on the space of forms defined by $\omega$.

\begin{definition}\label{def:F-op}
We define an operator $F_{\omega, \chi} : C^{\infty}(X) \rightarrow C^{\infty}(X) $ by 
$$F_{\omega, \chi} (\phi) = {-{ \left( \chi , {\sqrt{-1} \partial \bar{\partial} \phi} \right)_\omega} - \left( \partial \left( \Lambda_\omega \chi \right) , \bar{\partial} \phi \right)_\omega }. $$
\end{definition}

\begin{remark}
The operator $F_{\omega, \chi}$ becomes the linearization of the trace operator in $[\omega]$ at a solution of the $J_\chi$-equation by Equation \eqref{eq:linearisation}. 
\end{remark}

\begin{lemma}[{\cite[Lemma 2.2]{Hashimoto2018}}]\label{lem:properties}
The operator $F_ {\omega, \chi}$ is a complex self-adjoint second order elliptic linear operator. Moreover, it satisfies 
$$ \int_X {\phi \, F_{\omega, \chi} (\psi)} \, \omega^{n} 
= \int_X \nabla^{\bar{q}} \phi \, \nabla^{p} \psi \, {\chi_{p \bar{q}}} \, \omega^n.$$ 
In particular, the subspace $\mathop{\mathrm{ker}}F_{\omega, \chi}$ of $C^\infty(X)$ consists of constant functions on $X$.
\end{lemma}

\begin{proof}
For the readers' convenience, we recall the proof in \cite{Hashimoto2018} here.
Since $\chi$ is a K{\"a}hler form, the operator $F_{\omega, \chi}$ is a second order elliptic linear operator.
Recall that 
\begin{equation}\label{eq:sze}
\begin{aligned}
n \chi \wedge \omega^{n-1} &= (\Lambda_\omega \chi) \omega^n, \\ 
n(n-1) \chi \wedge \alpha \wedge \omega^{n-2} 
&= \left( \left( \Lambda_\omega \chi \right) \left( \Lambda_\omega \alpha \right) - \left( \chi, \alpha \right)_\omega \right) \omega^n
\end{aligned}
\end{equation}
for real $(1, 1)$-forms $\chi$ and $\alpha$ (see \cite[Lemma 4.7]{szekelyhidi2014introduction}).
For $\phi, \psi \in C^{\infty}(X, \mathbb{R}) $,
by using Equations \eqref{eq:sze} and integration by parts, we obtain 
\begin{align*}
& \int_X {\phi \, F_{\omega, \chi} (\psi)} \omega^{n} \\
=&\int_X \phi \left( n \left(n-1 \right) \chi \wedge {\sqrt{-1} \partial \bar{\partial} \psi} \right) \wedge \omega^{n-2} 
- \int_X \phi \, \left( \Delta_\omega \psi \right) \, (\Lambda_\omega \chi) \omega^n \\
& \qquad - \int_X n \, \phi \, \left( \sqrt{-1} \partial \left( \Lambda_\omega \chi \right) \wedge \bar{\partial}\psi \right) \wedge \omega^{n-1} \\
= &- \int_X n \left( n-1 \right) \left( \sqrt{-1} \partial \phi \wedge \bar{\partial} \psi \right) \wedge \chi \wedge \omega^{n-2}
 - \int_X n \, \phi \left( \Lambda_\omega \chi \right) \left( \sqrt{-1} \partial \bar{\partial} \psi \right) \wedge \omega^{n-1} \\
& \qquad - \int_X n \, \phi \left( \sqrt{-1} \partial \left( \Lambda_\omega \chi \right) \wedge \bar{\partial}\psi \right) \wedge \omega^{n-1} \\
= &- \int_X n \left( n-1 \right) \left( \sqrt{-1} \partial \phi \wedge \bar{\partial} \psi \right) \wedge \chi \wedge \omega^{n-2} 
+ \int_X n \left( \Lambda_\omega \chi \right) \left( \sqrt{-1} \partial \phi \wedge \bar{\partial} \psi \right) \wedge \omega^{n-1} \\
= &\int_X \left( \sqrt{-1} \partial \phi \wedge \bar{\partial} \psi , \chi \right)_\omega \omega^n.
\end{align*}
In the second equality, we used the fact that $\chi$ is closed.
Thus, the operator $F_{\omega, \chi}$ is a complex self-adjoint operator and $\mathop{\mathrm{ker}}F_{\omega, \chi}$ is precisely the set of constant functions since $\chi$ is a K{\"a}hler form.
\end{proof}

This implies that $F_{\omega, \chi}$ has index zero and is an isomorphism on the set of functions with integral zero. 

Let $X$ and $B$ be compact K{\"a}hler manifolds of dimensions $m + n$ and $n$ respectively and $\pi : X \rightarrow B$ be a holomorphic submersion of relative dimension $m$.
Suppose $\omega_X$ is a relatively K{\"a}hler metric, i.e., a closed $(1,1)$-form whose restriction $\omega_b$ to each fiber $X_b = \pi^{-1}(b)$ is a K{\"a}hler metric. Denote the vertical bundle $\mathop{\mathrm{ker}}d\pi$ by $\mathcal{V}$. 
Since $\omega_X$ is non-degenerate on a fiber, by setting 
$$ \mathcal{H}_x = \{ u \in T_x X \mid \omega_X (u, v)=0 \; \textrm{for all} \; v \in \mathcal{V}_x \},$$
we obtain the vertical-horizontal decomposition 
$$ TX = \mathcal{V} \oplus \mathcal{H}. $$
By this splitting, any tensor on $X$ decomposes into some terms via vertical and horizontal components. For a 2-tensor $A$, denote the purely vertical part by $A_\mathcal{V}$, the purely horizontal part by $A_\mathcal{H}$ and the mixed part by $A_m$, so $A =A_\mathcal{V} + A_m + A_\mathcal{H}$.

Given a $(p, p)$-form $\eta$ on $X$, we can integrate along fibers and associate a $(p-m, p-m)$-form $\pi_*(\eta)$ on $B$. Indeed, one reduces to the local case by a partition of unity argument on $B$. Take a sufficiently small neighborhood $U$ of $b \in B$ such that $ X|_U \cong Y \times U$ in smooth category. Then, locally on $\pi^{-1}(U)$,
a $2p$-form $\eta$ has two types of terms. 
The first type is the form of $\theta \wedge \pi^*\kappa$, where $\theta$ is a $2m$-form on $Y$ and $\kappa$ is a $2(p-m)$-form on $U$.
The second type is in the other cases. It is reduced to the case where we have only one term of the first type by linearity. 
Then, we define 
$$ \pi_*(\eta) = \left( \int_Y \theta \right) \kappa. $$
Moreover, since $\eta$ is a $(p,p)$-form and $\pi$ is a holomorphic submersion, the resulting form $\pi_*(\eta)$ is a $(p-m,p-m)$-form.
The following notation will be useful:

\begin{definition}\label{def:piB}
For a $(p,p)$-form $\eta$ on $X$, 
we define a $(p,p)$-form $\pi_B(\eta)$ on $B$ by 
$$ \pi_B(\eta)(b) =  V_b^{-1} \pi_*(\eta \wedge \omega_X^m)(b),$$ 
where $V_b = \int_{X_b} \omega_b^m$ is the volume of a fiber $X_b$. The volume $V_b$ is independent of $b \in B$ since $\omega_X$ is closed.
\end{definition}
In particular, for a function $f$ on $X$, 
$$\pi_B(f)(b) =  V_b^{-1} \int_{X_b} f \vert_{X_b} \, \omega_b^m. $$
Denote the subspace $\mathop{\mathrm{ker}}\Bigl( \pi_B : C^\infty \left( X \right) \rightarrow C^\infty \left( B \right) \Bigr)$ by $C_0^\infty \left( X \right)$. 
For any function $f$ on $X$, we have 
$ f = f_B + f_\mathcal{V},$
where $f_B = \pi^*\pi_B (f)$ and $f_\mathcal{V} = f - f_B \in C_0^\infty \left( X \right)$.

\begin{lemma}\label{lem:Kahler}
Let $\omega_X$ be a relatively K{\"a}hler form on $X$ and $\chi$ be a K{\"a}hler form on $X$. 
If $\omega_b$ is a solution of the $J_{\chi_b}$-equation on each fiber $X_b$ or the horizontal part of $\omega_X$ satisfies $\pi^*(\pi_B(\omega_X)_\mathcal{H}) = (\omega_X)_\mathcal{H}$, then $\pi_B(\chi_\mathcal{H})$ is a K{\"a}hler form on $B$.
\end{lemma}  
\begin{proof}
Note that
\begin{align*} 
\pi_B(\chi) &= V_b^{-1} \pi_*(\chi \wedge \omega_X^m) \\
&= V_b^{-1} \pi_*\left(m \, \chi_\mathcal{V} \wedge (\omega_X)^{m-1}_\mathcal{V} \wedge (\omega_X)_\mathcal{H} \right) + V_b^{-1}\pi_*(\chi_\mathcal{H} \wedge (\omega_X)^m_\mathcal{V}).
\end{align*}
The left-hand side is closed since $\chi$ and $\omega_X$ are closed. The first term of the second line is closed if $\omega_b$ is a solution of the $J_{\chi_b}$-equation on each fiber $X_b$ or $\pi^*(\pi_B(\omega_X)_\mathcal{H}) = (\omega_X)_\mathcal{H}$. Therefore, the second term, which is $\pi_B(\chi_\mathcal{H})$, is closed.
The $J_B$-invariance also follows for the same reason, where $J_B$ is the complex structure of $B$. 
Lastly we prove $\pi_B(\chi_\mathcal{H})$ is positive, that is, $\pi_B(\chi_\mathcal{H})(b)(u, J_B u) >0$
for a point $b \in B$ and a vector $u \in T_bB$. Take a local trivialization $\psi: X_U \cong Y \times U$ on a neighborhood $U$ of $b$, such that the splitting $TX \cong \mathcal{V} \oplus \mathcal{H}$ coincides with $TY \oplus TU$ on a fiber $X_b$ via $\psi_*$. By the definition of $\pi_B$, we have
\begin{align*}
\pi_B(\chi_\mathcal{H})(b)(u, J_B \, u) &=\frac{1}{V_b} \int_Y (\psi^{-1})^* (\chi_\mathcal{H} \wedge \omega^m_X)(b) (u, J_B u) \\ &= \frac{1}{V_b} \int_Y (\psi^{-1})^* (\chi_\mathcal{H})(b) (u, J_B u) (\psi^{-1})^* \omega^m_X
\\ &= \frac{1}{V_b} \int _Y \chi_\mathcal{H}(\psi^{-1}_*u, \psi^{-1}_* J_B u) (\psi^{-1})^*\omega_X^m
\\ &= \frac{1}{V_b}  \int _Y \chi\left(\psi^{-1}_*u, \psi^{-1}_* J_B u\right) (\psi^{-1})^*\omega_X^m.
\end{align*}
Note that, by the definition of $\mathcal{H}$, we have $J \psi_*^{-1} u \in \mathcal{H}$. Since we also have 
$\pi_*\left(\psi^{-1}_* J_B u-J\psi_*^{-1}u\right)=0,$
we get $\psi^{-1}_* J_B u=J\psi_*^{-1}u$. Thus the positivity of $\chi$ implies the positivity of $\pi_B(\chi_\mathcal{H})$.
\end{proof}

\subsection{Approximate solutions}\label{sec:approximation}
In this subsection, we construct a family of approximate solutions. We freely use the following notations.
Let $X$ and $B$ be compact K{\"a}hler manifolds of dimensions $m+n$ and $n$ respectively.
Let $\pi : X \rightarrow B$ be a holomorphic submersion of relative dimension $m$.
Fix a K{\"a}hler form $\chi$ of $X$. 
Suppose we have a relative K{\"a}hler form $\omega_X$ on $X$ such that the restriction $\omega_b$ on a fiber $X_b$ is a solution of the $J_{\chi_b}$-equation for all $b \in B$.
Split the tangent bundle $TX$ with respect to $\omega_X$.
Suppose moreover we have a solution $\omega_B$ of the $J_{\pi_B(\chi_\mathcal{H})}$-equation on $B$. Here, a $(1,1)$-form $\pi_B(\chi_\mathcal{H})$ is K{\"a}hler by Lemma \ref{lem:Kahler}.
We define $\omega_k = \omega_X + k\omega_B$.

We first construct a family of approximate solutions of the $J_\chi$-equation. More precisely, we prove the following:

\begin{proposition}\label{prop:app}
For any $r\in\mathbb{Z}_{\ge 0}$, 
there exist
$\{ \phi_{i, B} \}_{i=0}^r \in C^\infty (B)$ and 
$ \{ \phi_{i, \mathcal{V}} \}_{i=0}^r \in C_0^\infty \left( X \right)$ 
such that
$$ \omega_{k, r} = \omega_k + \sqrt{-1} \partial \bar{\partial} \left( \sum_{i=0}^r \phi_{i, B} k^{2-i} + \sum_{i=0}^r \phi_{i, \mathcal{V}} k^{-i} \right) $$ satisfies 
$$ \Lambda_{\omega_{k, r}} \left( \chi \right) = \sum_{i=0}^r k^{-i} {c_i} + O \left(k^{-r-1} \right), $$
where $ {c_i} $ are constants.
\end{proposition}
We will prove this by induction. The notation $O \left( k^{-r} \right)$ is only pointwise in this subsection.
We first see that $\omega_k$ is a solution of the $J_\chi$-equation up to $O \left(k^{-1} \right)$.
\begin{definition}
For a $(1,1)$-form $\chi$ on $X$, we define a function $ \Lambda_\mathcal{V} \chi$ on $X$ by
$$ \Lambda_\mathcal{V} \chi  
= m \, \frac{\chi_\mathcal{V} \wedge \omega^{m-1}_\mathcal{V}}{\omega^m_\mathcal{V}}.$$
Also, we define a function $\Lambda_{\omega_B} \chi$ on $X$ by
$$\Lambda_{\omega_B} \chi = n \, \frac{\chi_\mathcal{H} \wedge (\pi^* \omega_B)^{n-1}}{(\pi^* \omega_B)^n}.$$
\end{definition}
By a simple calculation, we have the following:

\begin{lemma}\label{lem:O(1)-term}
For a $(1,1)$-form $\chi$ on $X$, we have
$${\Lambda_{\omega_k}} \chi = \Lambda_\mathcal{V} \chi + k^{-1} {\Lambda_{\omega_B}} \chi + O \left(k^{-2} \right). $$
\end{lemma}

\begin{proof}
Since $\omega_X$ has no mixed term by the definition of the splitting,
\begin{align*}
&\Lambda_{\omega_k} \chi \\
=&\left( m+n \right) \frac{\chi \wedge \left( \omega_X + k \omega_B \right)^{m+n-1}}{\left( \omega_X + k \omega_B \right)^{m+n}} \\
=& \, m \frac
{ \chi_\mathcal{V} \wedge \left( \omega_X \right)_\mathcal{V}^{m-1} \wedge \Bigl( \left( \omega_X \right)_\mathcal{H} + k \omega_B \Bigr)^n}
{ \left( \omega_X \right)_\mathcal{V}^m \wedge \Bigl( \left( \omega_X \right)_\mathcal{H} + k \omega_B \Bigr)^n }
 + n \frac{ \chi_\mathcal{H}
  \wedge \left( \omega_X \right)_\mathcal{V}^m 
  \wedge \Bigl( \left( \omega_X \right)_\mathcal{H} + k \omega_B \Bigr)^{n-1}}
  { \left( \omega_X \right)_\mathcal{V}^m \wedge \Bigl( \left( \omega_X \right)_\mathcal{H} + k \omega_B \Bigr)^n} \\
=& \, \Lambda_\mathcal{V} \chi
   + n \frac{\chi_\mathcal{H} \wedge \left( {\displaystyle \sum_{i=0}^{n-1}} \begin{pmatrix} n-1 \\ i \end{pmatrix} k^{n-1-i} \left( \omega_X \right)_\mathcal{H}^i \wedge \omega_B^{n-1-i} \right)}
   { \left( {\displaystyle \sum_{i=0}^n} \begin{pmatrix} n \\i \end{pmatrix} k^{n-i} \left( \omega_X \right)_\mathcal{H}^i \wedge \omega_B^{n-i} \right)} \\
= & \, \Lambda_\mathcal{V} \chi + k^{-1} \Lambda_{\omega_B} \chi + O \left( k^{-2} \right).
\qedhere
\end{align*} 
\end{proof}

Since $\omega_b$ is a solution of the $J_{\chi_b}$-equation on each fiber, the lemma implies $\omega_k$ satisfies 
$$ \Lambda_{\omega_k} \chi = {c_0} + O(k^{-1}),$$ which is the case of $r = 0$ of Proposition \ref{prop:app}.

We will perturb $\omega_k$ to a solution of the $J_\chi$-equation up to $O(k^{-2})$ to prove the case of $r=1$ of Proposition \ref{prop:app}. 
Denote the $k^{-1}$-term $\Lambda_{\omega_B} \chi$ of $\Lambda_{\omega_k}\chi$
by $f_1$, and decompose $f_1 = \left( f_1 \right)_B + \left( f_1 \right)_\mathcal{V}$ as in the argument below Definition \ref{def:piB}. 
Since $\Lambda_{\omega_B}$ and $\pi_B$ are commutative by the definition of fiber integral as in Subsection \ref{sec:pre} and $\omega_B$ is a solution of $J_{\pi_B(\chi_\mathcal{H})}$-equation, the function $\left( f_1 \right)_B$ is constant. We denote this constant by $c_1$. Hence, we only need to delete the term $\left( f_1 \right)_\mathcal{V}$. 
This can be done by adding a fiberwise mean value zero function to $\omega_k$. 

\begin{definition}
For $\phi \in C^\infty(X)$, we define the operator $F_{ \left( \omega_X \right)_\mathcal{V}, \chi_\mathcal{V}}$ by 
$$\Bigl( F_{ \left( \omega_X \right)_\mathcal{V}, \chi_\mathcal{V}} (\phi) \Bigr) \Big\vert_{X_b} = F_{ \omega_b, \chi_b} \left( \phi\vert_{X_b} \right).$$
\end{definition}

\begin{lemma}\label{lem:k^-1}
 For $\phi \in C^\infty(X)$, we have
$$\frac{d}{dt}\Big|_{t = 0} \Lambda_{\omega_k + t \sqrt{-1} \partial \bar{\partial} \phi } (\chi) 
= F_{ \left( \omega_X \right)_\mathcal{V}, \chi_\mathcal{V}} \left(\phi \right) + O \left( k^{-1} \right).$$
\end{lemma}

\begin{proof}
As in Subsection \ref{sec:pre}, 
$$\frac{d}{dt}\Big|_{t = 0} \Lambda_{\omega_k + t \sqrt{-1} \partial \bar{\partial} \phi } (\chi) 
= {- {\left(g_k \right)^{m \bar{q}}} \, {\partial_m \partial_{\bar{n}} \phi} \, {\left( g_k \right)^{p \bar{n}}} \, {\chi_{p \bar{q}}}},$$
where $g_k$ is the metric tensor corresponding to $\omega_k$. 
We expand $g_k^{-1}$ in $k$ as a matrix in terms of vertical and horizontal parts, 
\begin{align*}
 {g_k}^{-1} &= \begin{bmatrix}
    \left( \omega_X \right)_\mathcal{V} & 0 \\
     0 & \left( \omega_X \right)_\mathcal{H} + k \omega_B   
\end{bmatrix}^{-1}
= \begin{bmatrix}
  \left( \omega_X \right)_\mathcal{V}^{-1} & 0 \\
  0 & k^{-1} \Bigl( k^{-1} \left( \omega_X \right)_\mathcal{H} + \omega_B \Bigr)^{-1}
\end{bmatrix} \\
&= \begin{bmatrix}
   \left( \omega_X \right)_\mathcal{V}^{-1} & 0 \\
   0 & 0
\end{bmatrix} 
+k^{-1} 
\begin{bmatrix}
0  & 0 \\ 0 & \omega_B^{-1}
\end{bmatrix} + O \left( k^{-2} \right).
\end{align*}
Therefore we see that 
\begin{align*}
{- {\left(g_k \right)^{m \bar{q}}}{\partial_m \partial_{\bar{n}} \phi} {\left( g_k \right)^{p \bar{n}}} {\chi_{p \bar{q}}}} 
&= {- {\left(g_X \right)_\mathcal{V}^{m \bar{q}}}\left( {\partial_m \partial_{\bar{n}} \phi}\right) {\left( g_X \right)_\mathcal{V}^{p \bar{n}}} {\chi_{p \bar{q}}}} + O \left(k^{-1} \right) \\ 
&= F_{ \left( \omega_X \right)_\mathcal{V}, \chi_\mathcal{V}} \left(\phi \right) + O \left( k^{-1} \right),
\end{align*}
where $\left(g_X \right)_\mathcal{V}$ the metric tensor on the vertical subbundle $\mathcal{V}$ of $TX$ corresponding to $\left(\omega_X\right)_\mathcal{V}$.
\end{proof}

By Lemma \ref{lem:properties}, there exists a unique function $\left( \phi_{1, \mathcal{V}} \right)_b \in C^\infty_0(X_b)$ for all $b \in B$  such that
$$F_{ \omega_b, \chi_b} \Bigl( \left( \phi_{1, \mathcal{V}} \right)_b \Bigr) = - \left( f_1 \right)_\mathcal{V}\vert_{X_b},$$
where $C^\infty_0(X_b)$ denotes the set of functions on $X_b$ with integral zero with respect to $\omega_X\vert_{X_b}$. 
Since the operator $F_{ \omega_X\mid_{X_b}, \chi\mid_{X_b}}$ is smooth in $b$, the implicit function theorem implies that there exists a unique function $\phi_{1, \mathcal{V}} \in C^\infty_0(X)$ such that 
$$ \phi_{1, \mathcal{V}}\mid_{X_b} = \left( \phi_{1, \mathcal{V}} \right)_b.$$
By combining this with Lemma \ref{lem:k^-1} and Lemma \ref{lem:O(1)-term}, we have 
\begin{align*}
\Lambda_{\omega_k + k^{-1} \sqrt{-1} \partial \bar{\partial} \phi_{1, \mathcal{V}}} (\chi)
&= \Lambda_\mathcal{V} \chi + k^{-1} \Bigl( c_1 + \left(f_1 \right)_\mathcal{V} + F_{ \left( \omega_X \right)_\mathcal{V}, \chi_\mathcal{V}} \left(\phi_{1, \mathcal{V}} \right) \Bigr) + O \left(k^{-2} \right) \\
&= c_0 + k^{-1} c_1 + O \left( k^{-2} \right).
\end{align*}
This completes the proof for the case of $r = 1$ of Proposition \ref{prop:app}.

To prove Proposition \ref{prop:app}, we repeat perturbing in the same manner. Suppose that we have the solution $\omega_{k, r}$ as in Proposition \ref{prop:app} for $r \ge 1$.
We need to calculate the linearization of the trace operator at $\omega_{k, r}$.

\begin{lemma}\label{lem:general_linearisation}
For a function $\phi \in C^\infty(X)$, we have
$$\frac{d}{dt}\Big|_{t = 0} \Lambda_{\omega_{k, r} + t \sqrt{-1} \partial \bar{\partial} \phi } (\chi) 
= F_{\left(\omega_X \right)_\mathcal{V}, \chi_\mathcal{V}} \left(\phi \right) + k^{-1} D_1 \left( \phi \right) + k^{-2} D_2 \left( \phi \right) + O \left( k^{-3} \right),$$
where the operators $D_1$ and $D_2$ satisfy
$D_1(\phi) = 0$ and $\pi_B \Bigl( D_2 (\phi) \Bigr) = F_{\omega_B, \pi_B(\chi_\mathcal{H})} (\phi)$
for $\phi \in C^\infty(B)$. Here, we suppress the pullbacks via $\pi$,
so for a function $\phi\in C^\infty(B)$ its pullback to $X$ is also denoted by $\phi$.
\end{lemma}

\begin{proof}
The proof is similar to the one of Lemma \ref{lem:k^-1}, but we need to deal with $g_{k,r}$ here instead of $g_k$. 
Denote $\Bigl( \displaystyle\sum_{i=0}^r k^{-i} \phi_{i, \mathcal{V}} \Bigr)$ by $\left( \phi_\mathcal{V} \right)_{k,r}$ and 
$\Bigl( \displaystyle\sum_{i=0}^r k^{-i} \phi_{i, B} \Bigr)$ by $\left( \phi_B \right)_{k,r}$ for simplicity. 
Then, we have
\begin{align*}
&{g_{k,r}}^{-1} \\ 
= &\begin{bmatrix}
    \left( \omega_k \right)_\mathcal{V} + \Bigl(\sqrt{-1} \partial \bar{\partial} \left( \phi_\mathcal{V} \right)_{k,r}\Bigr)_\mathcal{V}
    & \Bigl(\sqrt{-1} \partial \bar{\partial} \left( \phi_\mathcal{V} \right)_{k,r}\Bigr)_m  \\
     \Bigl(\sqrt{-1} \partial \bar{\partial} \left( \phi_\mathcal{V} \right)_{k,r}\Bigr)_m
     & \left( \omega_k \right)_\mathcal{H} + \Bigl(\sqrt{-1} \partial \bar{\partial} \bigl(\left( \phi_\mathcal{V} \right)_{k,r} + \left( \phi_B \right)_{k,r} \bigr)\Bigr)_\mathcal{H}
\end{bmatrix}^{-1}.
\end{align*}  
The $O\left( 1 \right)$-term of $g_{k, r}^{-1}$ is 
$\begin{bmatrix}
   \left( \omega_X \right)_\mathcal{V}^{-1} & 0 \\
   0 & 0
\end{bmatrix}$ in the same way as the proof of Lemma \ref{lem:k^-1}, since $\left( \phi_\mathcal{V} \right)_{k,r}$ is $O \left( k^{-1} \right)$. By using the formula of the inverse of a matrix and recalling that $\left( \phi_\mathcal{V} \right)_{k,r}$ is $O \left( k^{-1} \right)$, we see that the $k^{-1}$-term of $g_{k, r}^{-1}$ is 
$ \begin{bmatrix}
*  & 0 \\ 0 & \omega_B^{-1}
\end{bmatrix}$.
Let $\phi \in C^\infty(B)$.
As $\Bigl(\sqrt{-1}\partial\bar{\partial} (\pi^*\phi)\Bigr)_\mathcal{V} = \Bigl(\sqrt{-1}\partial\bar{\partial} (\pi^*\phi)\Bigr)_m =0$, the operator $D_1$ satisfies
$D_1 (\pi^*\phi) = 0$. 
The operator $D_2$ satisfies
\begin{align*}
\pi_B \Bigl( D_2(\phi) \Bigr) 
&= \pi_B \Bigl(- \bigl( \chi_\mathcal{H} , (\sqrt{-1} \partial \bar{\partial} \pi^*\phi)_\mathcal{H}\bigl)_{\omega_B} \Bigr) \\
&={- ( \pi_B(\chi_\mathcal{H}) , \sqrt{-1} \partial \bar{\partial} \phi)_{\omega_B} }\\
&= F_{\omega_B, \pi_B(\chi_\mathcal{H})} (\phi).
\qedhere
\end{align*}
\end{proof}

\begin{proof}[Proof of Proposition \ref{prop:app}]
Denote the $k^{-r-1}$-term of $\Lambda_{\omega_{k,r}} \chi$ by $f_{r+1} = \left( f_{r+1} \right)_B + \left( f_{r+1} \right)_\mathcal{V}$, where $\left( f_{r+1} \right)_B \in C^\infty(B)$ and $\left( f_{r+1} \right)_\mathcal{V} \in C^\infty_0(X)$.
Since $\omega_B$ is a solution of the $J_{\pi_B(\chi_\mathcal{H})}$-equation, Lemma \ref{lem:properties} implies that there exists a function $\phi_{r+1,B} \in C^\infty(B)$ such that $$F_{\omega_B,\pi_B(\chi_\mathcal{H})} \left( \phi_{r+1,B} \right) 
= -\left( f_{r+1} \right)_B + c_{r+1},$$ 
where $c_{r+1} = \Bigl(\int_B \omega_B^b \Bigr)^{-1}\Bigl( \int_B \left( f_{r+1} \right)_B \omega_B^b \Bigr)$. Hence Lemma \ref{lem:general_linearisation} implies that  we have
\begin{align*}
\Lambda_{\omega_{k, r} + k^{-r+1} \sqrt{-1} \partial \bar{\partial} \phi_{r+1, B}} (\chi) 
&= \sum_{i=0}^r k^{-i} {c_i} + k^{-r-1} \Bigl( f_{r+1} + D_2 \left( \phi_{r+1, B} \right) \Bigr) + O(k^{-r-2}) \\
&= \sum_{i=0}^r k^{-i} {c_i} + k^{-r-1} \Bigl( c_{r+1} + \left( f_{r+1} \right)'_\mathcal{V} \Bigr) + O(k^{-r-2})
\end{align*}
for some function $ \left( f_{r+1}\right)'_\mathcal{V} \in C^\infty_0(X)$. 
For the same reason stated below Lemma \ref{lem:k^-1}, there exists a function $\phi_{r+1, \mathcal{V}} \in C^\infty_0(X)$ such that 
$$ F_{\left( \omega_X \right)_\mathcal{V}, \chi_\mathcal{V}} 
\left( \phi_{r+1, \mathcal{V}} \right) 
= - \left( f_{r+1}\right)'_\mathcal{V}.$$
By using Lemma \ref{lem:general_linearisation} again, we obtain 
\begin{align*}
&\, \Lambda_{\omega_{k, r} + k^{-r+1} \sqrt{-1} \partial \bar{\partial} \phi_{r+1, B} + k^{-r-1} \sqrt{-1} \partial \bar{\partial} \phi_{r+1, \mathcal{V}}} (\chi) \\
=& \,  \sum_{i=0}^r k^{-i} {c_i} + k^{-r-1} \Bigl( c_{r+1} + \left( f_{r+1} \right)'_\mathcal{V}  + F_{\left( \omega_X \right)_\mathcal{V}, \chi_\mathcal{V}} 
\left( \phi_{r+1, \mathcal{V}} \right) \Bigr) + O(k^{-r-2})\\
=& \, \sum_{i=0}^{r+1} k^{-i} {c_i} + O(k^{-r-2}).
\end{align*}
This completes the proof of Proposition \ref{prop:app}.
\end{proof}

\subsection{First proof of Theorem \ref{thm:main}}\label{sec:firstproof}
We provide a proof of Theorem \ref{thm:main} by using the first approximation $\omega_{k,1}$ and the $\mathcal{C}$-subsolution criterion described below.

\begin{theorem}[{\cite{SW}}]\label{thm:SW}
    Let $(X,\chi)$ be an $n$-dimensional compact K{\"a}hler manifold and $\omega$ a K{\"a}hler form. Then the following are equivalent:
    \begin{enumerate}[font=\normalfont]
        \item There exists a solution $\omega'\in[\omega]$ of the $J_\chi$-equation.
        \item There exists a $\mathcal{C}$-subsolution $\omega'\in[\omega]$ of the $J_\chi$-equation, that is, the K{\"a}hler form $\omega'\in[\omega]$ satisfies $c\omega'^{n-1}-(n-1)\chi\wedge\omega'^{n-2}>0$.
    \end{enumerate}
\end{theorem}

\begin{proof}[Proof of Theorem \ref{thm:main}]
    By Theorem \ref{thm:SW}, it suffices to find a $\mathcal{C}$-subsolution in $[\omega_k]$ for $k\gg0$. We prove that $\omega_{k,1}$ is a $\mathcal{C}$-subsolution for $k\gg 0$.
    By a simple calculation, we can expand $c=\sum k^{-i}c_i$ as 
    \begin{equation}\label{eq:cexpan}
    \begin{split}
    c=&(m+n) \, \frac{\int_X \chi \wedge \omega_k^{m+n-1}}{\int_X \omega_k^{m+n}}\\
    =& \frac{m \int_X \chi_\mathcal{V} \wedge \left(\omega_X \right)_\mathcal{V}^{m-1} \wedge \left(\omega_k\right)_\mathcal{H}^n 
    +n \int_X \chi_\mathcal{H} \wedge \left(\omega_X\right)_\mathcal{V}^m \wedge \left( \omega_k\right)_\mathcal{H}^{n-1}}
    {\int_X \left(\omega_X\right)_\mathcal{V}^m \wedge \left( \omega_k\right)_\mathcal{H}^n}\\
    =& m \, \frac{\int_X \chi_\mathcal{V} \wedge \left(\omega_X\right)_\mathcal{V}^{m-1} \wedge \omega_B^n}{\int_X \left(\omega_X\right)_\mathcal{V}^m \wedge \omega_B^n} + k^{-1}\left(C_1 +n \, \frac{\int_X \left(\omega_X\right)_\mathcal{V}^m \wedge \chi_\mathcal{H} \wedge \omega_B^{n-1}}{\int_X \left(\omega_X\right)_\mathcal{V}^m \wedge \omega_B^n}\right) +O\left(k^{-2}\right),
    \end{split}
    \end{equation}
    where $C_1$ is the constant calculated as
    \begin{equation*}
    C_1 = \frac{\int_X \left(\Lambda_\mathcal{V} \chi\right) \left(\Lambda_{\omega_B} \omega_X\right) \omega_X^m \wedge \omega_B^n}{\int_X \omega_X^m \wedge \omega_B^n}-
    \frac{\int_X \left(\Lambda_\mathcal{V} \chi\right) \omega_X^m \wedge \omega_B^n}{\int_X \omega_X^m \wedge \omega_B^n}\, \frac{\int_X \left(\Lambda_{\omega_B} \omega_X\right) \omega_X^m \wedge \omega_B^n}{\int_X \omega_X^m \wedge \omega_B^n}.
    \end{equation*}
    The coefficient $c_0$ of the $k^0$-order term of \eqref{eq:cexpan} is the constant of the $J_{\chi_b}$-equation. Since $\omega_b$ is a solution of the $J_{\chi_b}$-equation on each fiber $X_b$, the constant $C_1$ above equals zero. Since we also see that
    \begin{equation}\label{eq:c_1}
    n \, \frac{\int_X\left(\omega_X\right)_\mathcal{V}^m \wedge \chi_\mathcal{H} \wedge \omega_B^{n-1} }{\int_X \left(\omega_X\right)_\mathcal{V}^m\wedge\omega_B^n}
    =n \, \frac{\int_B \pi_B\left(\chi_\mathcal{H}\right) \wedge \omega_B^{n-1}}{\int_B \omega_B^n},
    \end{equation}
    the coefficient $c_1$ of the $k^{-1}$-order term of \eqref{eq:cexpan} is the constant of the $J_{\pi_B(\chi_{\mathcal{H}})}$-equation.
    Fix a point $p$ in $X$. Let $\alpha$ be a non-zero semipositive form at $p$. We first calculate with $\omega_k$. At $p$, we have
    \begin{align*}
        &\left(c\omega_{k}^{m+n-1}-(m+n-1)\omega_{k}^{m+n-2}\wedge\chi\right)\wedge\alpha\\
        =&k^n \left(c_0 {{m+n-1}\choose n} \omega_X^{m-1}\wedge\omega_B^n-{m+n-1\choose n}(m-1)\omega_X^{m-2}\wedge\chi\wedge\omega_B^n \right)\wedge\alpha\notag\\
        &+k^{n-1}\left( c_0{m+n-1\choose n-1}\omega_X^m\wedge\omega_B^{n-1}-{m+n-1\choose n-1}m\omega_X^{m-1}\wedge\chi\wedge\omega_B^{n-1}\right)\wedge\alpha\notag\\
        &+k^{n-1}\left((c_1{m+n-1\choose n}\omega_X^{m-1}\wedge\omega_B^n\right)\wedge\alpha\notag\\
        &+k^{n-2}\left(c_0 {{m+n-1}\choose n-2} \omega_X^{m+1}\wedge\omega_B^{n-2}-{m+n-1\choose n-2}(m+1)\omega_X^{m}\wedge\chi\wedge\omega_B^{n-2}\right)\wedge\alpha\notag\\
        &+k^{n-2}\left(c_1 {{m+n-1}\choose n-1} \omega_X^{m}\wedge\omega_B^{n-1}\right)\wedge\alpha\notag\\
        &+k^{n-2}\left(c_2 {{m+n-1}\choose n} \omega_X^{m-1}\wedge\omega_B^{n}\right)\wedge\alpha+O(k^{n-3})\notag.
    \end{align*}
    We first consider the $k^n$-order term. Note that since $\omega_{\pi(p)}$ is a solution of the $J_{\chi_{\pi(b)}}$-equation on each fiber, it is a $\mathcal{C}$-subsolution of the $J_{\chi_{\pi(p)}}$-equation on each fiber. By the assumption that $\alpha_{\mathcal{V}}$ is a semipositive form at $p$, we see that the $k^n$-order term is positive if $\alpha_{\mathcal{V}}$ is non-zero.
    Let us perturb $\omega_{k}$ to $\omega_{k,1}:=\omega_k+\sqrt{-1}\partial\bar{\partial}k^{-1}\phi_{1,\mathcal{V}}$. Then, since this perturbation only contains $O(k^{-1})$-term, the leading term is unchanged and hence still positive if $\alpha_{\mathcal{V}}$ is non-zero.
    
    Now let us assume $\alpha_{\mathcal{V}}$ is zero. Note that since $\alpha$ is a semipositive form, if $\alpha_{\mathcal{V}}$ is zero, then the mixed term $\alpha_m$ is also zero. Indeed, if some coefficient  $\alpha(e_a,e_j)$ of the mixed term is not zero, then the determinant of the form restricted to the plane spanned by $e_a$ and $e_j$ will be negative, which contradicts $\alpha$ being semipositive. 
    In the above calculation, the first term of the $k^{n-1}$-order terms vanishes since $\alpha_\mathcal{V}$ is zero and $\omega_{\pi(p)}$ is a solution of the $J_{\chi_{\pi(p)}}$-equation. The second term of the $k^{n-1}$-order terms and the third term of the $k^{n-2}$-order terms also vanish since $\alpha_{\mathcal{V}}$ is zero. 
    When replacing $\omega_{k}$ with $\omega_{k,1}:=\omega_k+\sqrt{-1}\partial\bar{\partial}k^{-1}\phi_{1,\mathcal{V}}$, there is another $k^{n-1}$-order term coming from the combination of the $k^n$-order term in the above calculation and $k^{-1}\sqrt{-1}\bar{\partial}\partial\phi_{1,\mathcal{V}}$. However, this term still vanishes since $\alpha_\mathcal{V}=0$.
    By using the calculation in the proof of Lemma \ref{lem:Kahler} and the assumption that $\alpha=\alpha_{\mathcal{H}}$, we see that the $k^{n-2}$-order term becomes
    \begin{equation}\label{eq:push}
    \begin{split}
        &-k^{n-2}{m+n-1\choose n-1}(n-1)\omega_X^{m}\wedge\chi_{\mathcal{H}}\wedge\omega_B^{n-2}\wedge\alpha\\
        &+k^{n-2}{{m+n-1}\choose n-1}c_1 \omega_X^{m}\wedge\omega_B^{n-1}\wedge\alpha\\
        &+k^{n-2}{m+n-1\choose n-1}m\left( c_0\omega_X^{m-1}-(m-1)\omega_X^{m-2}\wedge\chi\right)\wedge\sqrt{-1}\partial\bar{\partial}\phi_{1,\mathcal{V}}\wedge\omega_B^{n-1}\wedge\alpha.
    \end{split}
    \end{equation}
    The function $\phi_{1,\mathcal{V}}$ satisfies $F_{(\omega_{X})_\mathcal{V},\chi_{\mathcal{V}}}(\phi_{1,\mathcal{V}})=-(f_1)_{\mathcal{V}}$, which implies the last term of \eqref{eq:push} equals 
    $$k^{n-2}{m+n-1\choose n-1}(f_1)_{\mathcal{V}}\, \omega_X^m\wedge\omega_B^{n-1}\wedge\alpha.$$ 
    Take local coordinates of $\mathcal{H}$ around $p$  
    in which $\omega_B=\sum\sqrt{-1}dz^i\wedge d\bar{z}^i$ and $\chi_{\mathcal{H}}=\sum\lambda_i\sqrt{-1}dz^i\wedge d\bar{z}^i$ at $p$. 
    Then, we see that
    \begin{align*}
        -(n-1)\chi_{\mathcal{H}}\wedge\omega_B^{n-2}+\big(c_1+(f_1)_{\mathcal{V}}\big)\omega_B^{n-1}
        =&-(n-1)\chi_{\mathcal{H}}\wedge\omega_B^{n-2}+(\Lambda_{\omega_B}\chi)\omega_B^{n-1}\\
        =&\sum_i\left(\lambda_i\prod_{j\neq i}\sqrt{-1}dz^j\wedge d\bar{z}^j\right)>0,
    \end{align*}
    where the first equality follows from the definition of $(f_1)_\mathcal{V}$ (see the arguments below Lemma \ref{lem:O(1)-term}), and the last positivity follows from that $\chi_{\mathcal{H}}$ is a positive form by the calculation in the proof of Lemma \ref{lem:Kahler}. Therefore, in either case where $\alpha_{\mathcal{V}}$ are zero or not, we have proved that $(c\omega_{k,1}^{m+n-1}-(m+n-1)\omega_{k, 1}^{m+n-2}\wedge\chi)\wedge\alpha$ is positive for $k\gg0$. Since the arguments above do not depend on $p$ or $\alpha$, we have proved that the form $c\omega_{k,1}^{m+n-1}-(m+n-1)\omega_{k, 1}^{m+n-2}\wedge\chi$ is positive for $k\gg0$.
\end{proof}

\subsection{Second proof of Theorem \ref{thm:main}}\label{sec:proof}
We perturb the approximate solution $\omega_{k, r}$ obtained in the previous subsection to a genuine solution of the $J_\chi$-equation for $k \gg 0$. A positive integer $r \gg 0$ is fixed and we consider all $k \gg 0$. An important ingredient is the following theorem called the quantitative inverse function theorem \cite[Theorem 4.1]{Fine}.

\begin{theorem}\label{thm:IFT} We assume the following conditions. \leavevmode
\begin{itemize}
 \item A map $F : B_1 \rightarrow B_2$ is a differentiable map of Banach spaces, whose derivative at 0, $DF$, is an isomorphism of Banach spaces, with inverse $P$.
 \item  A constant $\delta' $ is the radius of the closed ball in $B_1$, centered at 0, on which $F-DF$ is Lipschitz, with constant $1/ \left(2 || P || \right)$. 
 \item A constant $\delta$ is defined by $\delta = \delta' /(2 ||P||)$.
\end{itemize}
Then, whenever $y \in B_2$ satisfies $||y - F(0) || < \delta$, there exists $x \in B_1$ such that $F(x) = y$. Moreover, such an $x$ is unique subject to the constraint $||x|| < \delta'$.
\end{theorem}

In this subsection, if it is not stated explicitly, all Sobolev spaces are considered with respect to $g_{k,r}$. Let $l$ be a positive integer satisfying $l-(m+n) > 0$. Then, for a function $\phi \in L^2_{l+2}$,  a $(1,1)$-form $\sqrt{-1} \partial \bar{\partial} \phi$ is continuous by the Sobolev embedding. Since the trace operator is analytic in the metric, we can extend the trace operator to the smooth operator 
$\Lambda_{\omega_{k, r}}(\chi) : L^2_{l+2} \rightarrow L^2_l$ defined by 
$\Lambda_{\omega_{k, r}}(\chi)(\phi) = \Lambda_{\omega_{k, r} + \sqrt{-1} \partial \bar{\partial} \phi}(\chi)$.  
Denote the set of functions in $L^2_l$ with integral zero with respect to $\omega_{k,r}$ by $L^2_{l,0}$,  the projection from $L^2_l$ to $L^2_{l,0}$ by $p$, and define $\mathcal{L}_{k,r} : = p \circ \Lambda_{\omega_{k, r}}(\chi)|_{L^2_{l+2,0}}$.
We will use Theorem \ref{thm:IFT} for the operator $\mathcal{L}_{k,r}$. 
Therefore, we need to check that the linearization $D_0 \mathcal{L}_{k,r}$ at 0 is isomorphism and that the norm of the inverse and the radius of the ball as in Theorem \ref{thm:IFT} can be estimated sufficiently to conclude 0 and $\mathcal{L}_{k,r}(0)$ are close enough.
First, we show that $D_0 \mathcal{L}_{k,r}$ is an isomorphism and estimate the norm of the inverse. The following lemma \cite[Lemma 6.10]{Fine} is key:

\begin{lemma}\label{lem:inverse}
Let $D:B_1\rightarrow B_2$ be a bounded invertible linear map of Banach spaces with bounded inverse $Q$. If $L:B_1\rightarrow B_2$ is another linear map with 
$$||L-D|| \le \left( 2||Q|| \right)^{-1},$$
then $L$ is also invertible with bounded inverse $P$ satisfying $||P|| \le 2||Q||$.
\end{lemma}

The operator close to $D_0 \mathcal{L}_{k,r}$ will be $F_{\omega_{k,r}, \chi}$, which is an isomorphism between $L^2_{l+2,0}$ and $L^2_{l,0}$ by Lemma \ref{lem:properties}

\begin{lemma}\label{lem:inverse_estimate}
There exist a constant $C$, which depends on $r$, and an integer A, which is independent of $r$, 
such that for all $\phi \in L^2_{l,0}$ and $k \gg 0$, we have
$$ ||F_{\omega_{k,r}, \chi}^{-1} (\phi)||_{L^2_{l+2}}
\le Ck^A ||\phi||_{L^2_l}.$$
\end{lemma}

\begin{proof}
By Lemma \ref{lem:properties}, for $\phi \in L^2_{l, 0}$, we have 
$$\int_X \phi \, F_{\omega_{k,r}, \chi}(\phi) \, \omega_{k,r}^{m+n}
= \int_X \left( \nabla^{\bar{q}}_{k,r} \phi \right) \left( \nabla^p_{k,r} \phi \right) \chi_{p \bar{q}} \, \omega_{k,r}^{m+n},$$
where $\nabla_{k,r}$ denotes the covariant derivative with respect to $\omega_{k,r}$.
Since the leading term of $\omega_{k,r}$ is $k \omega_B$ (recall that we fix $r$) and $\chi$ is a K{\"a}hler form, there exists a constant $C_1$ such that for all $k>0$,
$$ \chi \ge C_1 k^{-1} \omega_{k,r}.$$ 
Therefore, if we denote the first non-zero eigenvalue of $F_{\omega_{k,r}, \chi}$  by $\lambda_1$ and an eigenfunction corresponding to $\lambda_1$  by $\phi_1$, we have
\begin{align*}
\lambda_1 &= \left( \int_X {\phi_1} \, F_{\omega_{k,r}, \chi} \left( \phi_1 \right) \, \omega_{k,r}^{m+n} \right) \Big/ {||\phi_1||^2_{L^2}} \\
&=\left( \int_X \left( \nabla^{\bar{q}}_{k,r} \phi_1 \right) \left( \nabla^p_{k,r} \phi_1 \right) \chi_{p \bar{q}} \, \omega_{k,r}^{m+n} \right) \Big/ {||\phi_1||^2_{L^2}} \\
& \ge C_1 k^{-1} \left( \int_X |\nabla_{k,r} \phi_1|^2 \, \omega_{k,r}^{m+n} \right) \Big/ {||\phi_1||^2_{L^2}} \\
& \ge C_2 k^{-2}.
\end{align*}
In the last inequality, we used \cite[Lemma 6.5]{Fine}. Also, for $\phi \in L^2_{l+2,0}$, by the following lemma, we have
$$ ||\phi||_{L^2_{l+2}} \le C k^A\left( ||\phi||_{L^2} + ||F_{\omega_{k,r}, \chi}(\phi)||_{L^2_l} \right).$$
Combining these two estimates gives the desired estimate.
\end{proof}

\begin{lemma}\label{elliptic}
There exist a constant C and an integer $A$, independent of $k$ and $r$, 
such that for $\phi \in L^2_{l+2}$, an integer $r \ge A$ and any $k \gg 0$
we have
$$ ||\phi||_{L^2_{l+2}} \le C k^A\left( ||\phi||_{L^2} + ||F_{\omega_{k,r}, \chi}(\phi)||_{L^2_l} \right).$$
\end{lemma}
\begin{remark}
Note that our operator $F_{\omega_{k,r},\chi}$ depends on a reference metric $\chi$ which is independent of $k$. Thus, when we try to apply the arguments in \cite[Sections 5 and 6]{Fine}, the constant of the elliptic estimate on a model space depends on $k$ in our case. Note also that a model space approximates the original space in the order $k^{-1/2}$ at best. Hence, the constant of the elliptic estimate in a model space increases faster than the approximation by a model space.
Due to this, we cannot deduce the estimate in a local fibered space $X|_D$ from the estimate on a model space $S \times D$, where $S$ is a fiber and $D$ is a small neighborhood on $B$ as in \cite[Lemma 5.9]{Fine}. 
Instead of the model space $S\times D$, 
we use product coordinate neighborhoods. We first obtain the elliptic estimate in the above form on product coordinates by the standard proof of the elliptic estimate on the Euclidean space, and then patch them by using a partition of unity on the total space $X$.
\end{remark}
\begin{proof}
Choose a finite product coordinate system $\{(U_i,\psi_i; z^1, \dots , z^{m+n})\}^N_{i=1}$ on the total space $X$ such that $z^j$ are base direction for $j=m+1, \dots, m+n$ and the coordinate can extend to $U'_i$ satisfying $U_i \subset \subset U'_i$.
Take also a partition of unity $\{ \rho_i \}^N_{i=1}$ subordinate to 
$\{(U_i,\psi_i)\}^N_{i=1}$. Define a new coordinate system $\{ (U_i, \tau_i; w^1,\dots, w^{m+n}) \}^N_{i=1}$, which is given by $w^j=z^j$ for $j=1, \dots, m$ and $w^j= \sqrt{k} z^j$ for $j=m+1, \dots, m+n$. By this scaling and the construction of the approximate solutions $\omega_{k,r}$, the coefficients of the corresponding metric $g_{k,r}$ are $O(1)$ in $C^t$ on this coordinate $(w^1, \dots, w^{m+n})$ for $k \gg 0$ and each fixed $r$, where $t$ is chosen large enough and fixed for the arguments below.  Denote the linearization of the trace operator by $G_{k,r}$, i.e. 
$$G_{k,r}(\phi) = - \left( \chi, \sqrt{-1} \partial\bar{\partial} \phi \right)_{\omega_{k,r}}.$$
We will prove the elliptic estimate in the statement for the operator $G_{k,r}$ first. The following claim is the local version of the estimate in the statement:
\begin{claim*}
On $U_i$, denote the corresponding operator of $G_{k,r}$ via $\tau_i$ on $\tau_i(U_i)$ by $G_{k,r}$ as well. For a real number $s \in \mathbb{R}$ and a function $\phi \in L^2_{s+2}$ with compact support contained in $U_i$, there exist a constant $C$ and $A$, depending on $s$ but not on $k$ or $r$, such that 
$$ ||\phi||_{L^2_{s+2}(\mathbb{R}^{2(m+n)})} \le C k^A \left( ||\phi||_{L^2_{s+1}(\mathbb{R}^{2(m+n)})} + ||G_{k,r}(\phi)||_{L^2_s(\mathbb{R}^{2(m+n)})} \right),$$
where we suppress $\tau_i$ and consider functions on $U_i$ as on $\tau_i(U_i)$.
\end{claim*}
\begin{proof}[Proof of Claim]

We first observe that the principal symbol of $G_{k,r}$ can be estimated from below by $C/k$ for some constant $C$ on $\tau_i(U_i)$. On a coordinate neighborhood $\tau_i(U_i)$, we have
$$g_{k,r}^{-1}\chi g_{k,r}^{-1} \rightarrow  \begin{bmatrix} (g_X)_{11}^{-1} \chi_{11} (g_X)_{11}^{-1} & 0 \\ 0& 0\end{bmatrix} \text{as $k \rightarrow \infty$},$$
where we consider metrics $g_{k,r}$ and $\chi$ as real $2(m+n) \times 2(m+n)$ matrices on $\tau_i(U_i)$ and $(g_X)_{11}$ and $\chi_{11}$ are matrices of purely fiber part.
This convergence is uniformly on $\tau_i(U_i)$. Thus, eigenvalues of the principal symbol $g_{k,r}^{-1}\chi g_{k,r}^{-1}$ converge to $2m$ positive numbers and $2n$ zeros. For the eigenvalue $\lambda$ which converges to zero, consider $k \lambda$. This is the zero of the polynomial
$$P(t)= \det{\left(
\begin{bmatrix}
    I & 0 \\ 0 & \sqrt{k}I
\end{bmatrix}
g_{k,r}^{-1} \chi g_{k,r}^{-1}
\begin{bmatrix}
    I & 0 \\ 0 & \sqrt{k}I
\end{bmatrix}
- \begin{bmatrix}
    \lambda I & 0 \\ 0 & t I
\end{bmatrix}\right)}.$$
We have
$$\begin{bmatrix}
    I & 0 \\ 0 & \sqrt{k}I
\end{bmatrix}
g_{k,r}^{-1} \chi g_{k,r}^{-1}
\begin{bmatrix}
    I & 0 \\ 0 & \sqrt{k}I
\end{bmatrix} \rightarrow 
\begin{bmatrix}
    (g_X)^{-1}_{11} & 0 \\ D & g_B^{-1}
\end{bmatrix}
\chi'
\begin{bmatrix}
    (g_X)^{-1}_{11} & {}^t\! D \\ 0 & g_B^{-1}
\end{bmatrix} \text{as $k \rightarrow \infty$},$$
where $D$ is some matrix and $g_B$ and $\chi'$ denote the corresponding matrices of the metrics $\omega_B$ and $\chi$, respectively, in the coordinate neighborhood $\psi_i(U_i)$ which is independent of $k$. 
The convergence is uniform in $\tau_i(U_i)$, and thus $k \lambda$ converges to some positive number, since $g_X$, $g_B$ and $\chi'$ are positive in each part and $\lambda$ converges to zero. This confirms that we can estimate the eigenvalues of the principal symbol of $G_{k,r}$ from below by $C/k$. 
We then follow the arguments in \cite[Section 3.4]{Donaldson}, also used in \cite[Lemma 5.4]{Fine}. We divide the Euclidean space $\mathbb{R}^{2(m+n)}$ by rectangles $B_a:=\prod_{i=1}^{2(m+n)}[a_i,a_i+1)$, where $a:=(a_1,a_2,\dots,a_{2(m+n)})\in\mathbb{Z}^{2(m+n)}$. Also, denote by $B_a^+$ a rectangle slightly larger than $B_a$. Then, by the proof of the standard elliptic estimate (such as \cite[Chapter 5 Theorem 11.1]{Taylor}), we have 
$$ ||\phi||_{L^2_{s+2}(B_a)} \le C k^A \left( ||\phi||_{L^2_{s+1}(B_a^+)} + ||G_{k,r}(\phi)||_{L^2_s(B_a^+)} \right),$$
where $C$ and $A$ are constants, depending on $s$ but not on $k$ or $r$. Here, the constants $C$ and $A$ can also be independent of $a$, since the symbol and coefficients of $G_{k,r}$ can be uniformly estimated in $\tau_i(U_i)$ as above. Therefore, we see that
\begin{align*}
    ||\phi||_{L^2_{s+2}(\mathbb{R}^{2(m+n)})}
    =&\sum_{a\in \mathbb{Z}^{2(m+n)}}||\phi||_{L^2_{s+2}(B_a)}\\
    \le &\sum_{a\in \mathbb{Z}^{2(m+n)}}C k^A \left( ||\phi||_{L^2_{s+1}(B_a^+)} + ||G_{k,r}(\phi)||_{L^2_s(B_a^+)}\right)\\
    \le &C k^A \left( ||\phi||_{L^2_{s+1}(\mathbb{R}^{2(m+n)})} + ||G_{k,r}(\phi)||_{L^2_s(\mathbb{R}^{2(m+n)})} \right),
\end{align*} 
where the last inequality follows since each $B_a^+$ overlaps only finitely many times, independently of $a$.
\end{proof}
Let $\phi \in L^2_{l+2}$. Then, we have 
\begin{align*}
||\phi||_{L^2_{l+2}} &\le \sum^N_{i=1} ||\rho_i \phi||_{L^2_{l+2}}\\
                         &= \sum^N_{i=1} \left(\sum^{l+2}_{j=0} \int_{\tau_i(U_i)} |\nabla^j \rho_i \phi|^2_{g_{k,r}} \, \left(\det[g_{k,r}]\right) \, dw \right)^{1/2}\\
                         &\le C \sum^N_{i=1} \left(\sum^{l+2}_{j=0} \int_{\tau_i(U_i)} |D^j \rho_i \phi|^2 \, dw \right)^{1/2},
\end{align*}
where $\nabla^j$ is the $j$-th covariant derivative with respect to $g_{k,r}$ and $D^j$ is the $j$-th derivative with respect to the coordinate $(w^1, \dots, w^{m+n})$. The last inequality follows from the fact that the corresponding matrices $g_{k,r}$ and $g_{k,r}^{-1}$  with respect to the coordinate $(w^1, \dots, w^{m+n})$ are $O(1)$ in $C^t(\tau_i(U_i))$ for $k \gg 0$ and each fixed $r$. Using the claim above, we get 
$$||\phi||_{L^2_{l+2}} \le C k^{A} \sum^N_{i=1} \left(||\rho_i \phi||_{L^2_{l+1}(\tau_i(U_i))} + ||G_{k,r}(\rho_i \phi)||_{L^2_{l}(\tau_i(U_i))} \right).$$
Define the operator $[G_{k,r},\rho_i]:= G_{k,r} \circ \rho_i  -\rho_i \circ G_{k,r}$, where $\rho_i$ is the operator of multiplication by $\rho_i$. This operator is a differential operator of the first order. Moreover, the coefficients can be estimated from above by constants in $C^l(\tau_i(U_i))$
Therefore, by introducing a new cutoff function $\rho_{i,1}$, which is constant 1 on $\mathrm{Supp}(\rho_i)$ and has compact support $\mathrm{Supp}(\rho_{i,1})$ contained in $U_i$, we have
$$||\phi||_{L^2_{l+2}} \le C k^{A} \sum^N_{i=1} \left(||\rho_{i,1} \phi||_{L^2_{l+1}(\tau_i(U_i))} + ||\rho_i G_{k,r}(\phi)||_{L^2_{l}(\tau_i(U_i))} \right).$$
For the second term, note that the derivatives of $\rho_i$ can be estimated from above, and $L^2_l(\tau_i(U_i))$-norm is equivalent to $L^2_l$-norm as $g_{k,r}$ is $O(1)$ on this coordinate. Thus, we get
$$||\phi||_{L^2_{l+2}} \le C k^{A} \left(\sum^N_{i=1} ||\rho_{i,1} \phi||_{L^2_{l+1}(\tau_i(U_i))}\right) + C k^A ||G_{k,r}(\phi)||_{L^2_l}.$$
We next estimate $||\rho_{i,1} \phi||_{L^2_{l+1}(\tau_i(U_i))}$ by using the claim above and iterate this process. Eventually, we get
$$||\phi||_{L^2_{l+2}} \le C k^{A} \left(||\phi||_{L^2} + ||G_{k,r}(\phi)||_{L^2_l}\right),$$
where we use the same notation $A$ for simplicity.
To obtain the desired estimate for $F_{\omega_{k,r},\chi}$, recall that we have
$$\Bigl( G_{k,r} - F_{\omega_{k,r}, \chi} \Bigr)(\phi)
= \Bigl( \partial \Lambda_{\omega_{k,r}}(\chi), \bar{\partial}\phi \Bigr)_{\omega_{k,r}}.$$
Although we only proved that $\omega_{k,r}$ is an approximate solution of the $J_\chi$-equation up to $O(k^{-r-1})$ pointwisely in the last subsection, the argument of \cite[Lemma 5.7]{Fine} follows and we have
\begin{align*}
p(\Lambda_{\omega_{k,r}}\chi)&= O(k^{-r-1}) &\text{in $C^l(g_{k,r})$},\\
p(\Lambda_{\omega_{k,r}}\chi)&= O(k^{-r-1-n/2}) &\text{in $L^2_l(g_{k,r})$}.
\end{align*}
Therefore, we have
$$||G_{k,r}-F_{\omega_{k,r},\chi}|| \le C k^{-r-1}$$
and if we choose $r \ge A$, we get
$$||\phi||_{L^2_{l+2}} \le C k^{A} \left(||\phi||_{L^2} + ||F_{\omega_{k,r},\chi}(\phi)||_{L^2_l}\right)$$
for $k \gg 0$. 
\end{proof}

Recall that we have 
$$\Bigl( D_0 \Lambda_{\omega_{k,r}}(\chi) - F_{\omega_{k,r}, \chi} \Bigr)(\phi)
= \Bigl( \partial \Lambda_{\omega_{k,r}}(\chi), \bar{\partial}\phi \Bigr)_{\omega_{k,r}}.$$
As $\omega_{k,r}$ is a solution of the $J_\chi$-equation up to $O \left( k^{-r-1} \right)$ and the derivative $D_0 \mathcal{L}_{k,r}$ is given by $D_0 \mathcal{L}_{k,r} = p \, \circ \, \Bigl( D_0 \Lambda_{\omega_{k,r}}(\chi) \Bigr)$, the same arguments in the proof of \cite[Theorem 6.1] {Fine} imply that we have
$$ \left\Vert \left( D_0 \mathcal{L}_{k,r} - F_{\omega_{k,r}, \chi} \right) (\phi)\right\Vert_{L^2_l}
\le c k^{-r-1} ||\phi||_{L^2_{l+2}}.$$
Hence, if we choose $r \ge A$, Lemma \ref{lem:inverse} implies that $D_0 \mathcal{L}_{k,r}$ is an isomorphism for $k\gg0$ and the operator norm of the inverse $P$ satisfies $||P||_\mathrm{op} \le C k^A$
for some constant $C$.

We also estimate the radius of the ball on which $\mathcal{L}_{k,r} - D_0 \mathcal{L}_{k,r}$ is Lipschitz with constant $1/\left( 2||P|| \right)$.
 Denote the nonlinear part of $\mathcal{L}_{k,r}$ by $\mathcal{N}_{k,r}$, i.e., $\mathcal{N}_{k,r} = \mathcal{L}_{k,r} - D_0 \mathcal{L}_{k,r}$. Then, by the mean value theorem,
 $$ \left\Vert \mathcal{N}_{k,r} (\phi) - \mathcal{N}_{k,r} (\psi) \right\Vert_{L^2_l} 
\le 
\sup_{f \in [\phi, \psi]} \Vert D_f \mathcal{N}_{k,r} \Vert
\left\Vert \phi - \psi \right\Vert_{L^2_{l+2}},$$
where $D_f \mathcal{N}_{k,r}$ denotes the derivative of $\mathcal{N}_{k,r}$ at $f$.

\begin{lemma}\label{lem:lipschitz}
There exists a constant $C$ which is independent of $k$ such that if 
$ \Vert f \Vert_{L^2_{l+2}} \le \epsilon \, k^{-n/2}$ for a small constant $\epsilon$, 
we have 
$$\Vert D_f \mathcal{N}_{k,r} \Vert
\le 
C k^{n/2} ||f||_{L^2_{l+2}}.$$
\end{lemma}

\begin{proof}
 For $\phi \in L^2_{l+2,0}$, by a simple calculation, we have
\begin{align*}
\Vert \left( D_f \mathcal{N}_{k,r} \right) (\phi) \Vert_{L^2_l}
&= \Vert \left( D_f \mathcal{L}_{k,r} - D_0 \mathcal{L}_{k,r} \right) (\phi) \Vert_{L^2_l}\\
&\le \Vert -\left( \chi, \sqrt{-1} \partial \bar{\partial} \phi \right)_{\omega_{k,r}+\sqrt{-1} \partial \bar{\partial} f} + \left(\chi, \sqrt{-1} \partial \bar{\partial} \phi \right)_{\omega_{k,r}} \Vert_{L^2_l} \\
&= \Big\Vert \left( \partial_p \partial_{\bar{q}} \phi \right) \, \chi_{m \bar{n}} \, \Bigl( -\left(g_{k,r,f} \right)^{m \bar{q}} \left(g_{k,r,f} \right)^{p \bar{n}} + g_{k,r}^{m \bar{q}} \, g_{k,r}^{p \bar{n}} \Bigr) \Big\Vert_{L^2_l},
\end{align*}
where $g_{k,r,f}$ is the metric tensor corresponding to $\omega_{k,r} + \sqrt{-1} \partial \bar{\partial} f$.
Note that the Sobolev constants with respect to $g_{k,r}$ are independent of $k$ by \cite[Lemma 5.8]{Fine}. If $2l > 2(m+n)$, by the Sobolev inequality, for any tensors $T$ and $T'$ we have
$\Vert T \cdot T' \Vert_{L^2_l} \le C \Vert T \Vert_{L^2_l} \Vert T' \Vert_{L^2_l}$ for some constant $C$ which is independent of $k$, where $T \cdot T'$ denotes tensor product or contraction. Thus, we have
$$\Vert \left( D_f \mathcal{N}_{k,r} \right) (\phi) \Vert_{L^2_l}
\le
C \Vert \phi \Vert_{L^2_{l+2}} \, \Vert \chi \Vert_{C^l} \, \Vert - g^{-1}_{k,r,f} \otimes g^{-1}_{k,r,f} + g^{-1}_{k,r} \otimes g^{-1}_{k,r} \Vert_{L^2_l}.
$$
If we assume $ \Vert f \Vert_{L^2_{l+2}} \le \epsilon \, k^{-n/2}$ for a small constant $\epsilon$, 
since $\Vert g^{-1}_{k,r} \Vert_{L^2_l} < C k^{n/2}$ for some constant $C$ and 
$-\left(g_{k,r,f}\right)^{m \bar{q}} + g_{k,r}^{m \bar{q}} 
= \left(g_{k,r,f}\right)^{m \bar{p}} \left( \partial_n \partial_{\bar{p}} f \right) g_{k,r}^{n \bar{q}},
$
 we have
$$ \Vert g_{k,r,f}^{-1} \Vert_{L^2_{l+2}} 
\le \Vert g_{k,r,f}^{-1} - g_{k,r}^{-1} \Vert_{L^2_{l+2}} +C \, k^{n/2}
\le \Vert g_{k,r,f}^{-1} \Vert_{L^2_{l+2}} \epsilon \, C +C \, k^{n/2}.
$$
Also, the form $\chi$ is uniformly bounded above with respect to the $C^l$-norm (\cite[Lemma 5.6]{Fine}). 
By combining these estimates, we have the desired estimate.
\end{proof}

This implies that $\mathcal{L}_{k,r} - D_0\mathcal{L}_{k,r}$ is Lipschitz with constant $1/\left(2\Vert P \Vert\right)$ on the ball centered at $0$ with radius $C \, k^{-A-n/2}$ for some constant $C$. 
As $\omega_{k,r}$ is an approximate solution of the $J_\chi$-equation, the same arguments as in \cite[Lemmas 5.6 and 5.7]{Fine} imply that we have $\mathcal{L}_{k,r}(0) = O\left(k^{-r-1} \right)$ in $C^l(g_{k,r})$ and $\Vert \mathcal{L}_{k,r}(0) \Vert_{L^2_l} = O \left(k^{-r-1+{n/2}} \right)$.
Therefore, if we choose $r \ge  2A + n $, for all $k \gg 0$, we have a function $\phi$ such that $\mathcal{L}_{k,r}(\phi) = 0$. As $\mathcal{L}_{k,r}$ is an elliptic operator of second order, if we make $l$ large enough, the regularity theorem implies $\phi \in C^\infty$. This completes the proof of Theorem \ref{thm:main}. 

\subsection{Proof of Theorem \ref{thm:converse}}\label{sec:converse}

We also consider the converse implication of Theorem \ref{thm:main}. Instead of considering a solution of the $J$-equation, we consider the topological condition called \textit{$J$-positivity} and it relates by the following theorem (\cite{Song}):

\begin{theorem}[{\cite[Corollary 1.2]{Song}}]\label{J-positivity}
Fix a K{\"a}hler manifold $X$ of dimension $n$ with K{\"a}hler metrics $\omega$ and $\chi$. 
Let $c > 0$ be the constant determined by 
$$c \int_X \omega^n = n \int_X \chi \wedge \omega^{n-1}.$$
Then, there exists a solution of $J_\chi$-equation in the class $[\omega]$ if and only if 
we have
$$\int_W c\,\omega^p - p \chi \wedge \omega^{p-1} > 0$$
for all $p$-dimensional subvarieties W with $p = 1,2,\dots,{n-1}.$
\end{theorem} 
The following definition is from \cite[Definition 1.1]{Song}:

\begin{definition}
Let $X$ be a K{\"a}hler manifold of dimension $n$ with K{\"a}hler metrics $\omega$ and $\chi$.
\begin{enumerate}
\item
The pair $([\omega],[\chi])$ is said to be \textit{$J$-positive} if we have
$$ p \, \frac{\int_W \chi \wedge \omega^{p-1}}{\int_W \omega^p} <
n \, \frac{\int_X \chi \wedge \omega^{n-1}}{\int_X \omega^n}$$
for any $p$-dimensional subvarieties $W$ of $X$ with $p = 1,2,\dots,{n-1}$.
\item
The pair $[(\omega],[\chi])$ is said to be \textit{$J$-nef} if we have
$$ p \, \frac{\int_W \chi \wedge \omega^{p-1}}{\int_W \omega^p} \le
n \, \frac{\int_X \chi \wedge \omega^{n-1}}{\int_X \omega^n}$$
for any $p$-dimensional subvarieties $W$ of $X$ with $p = 1,2,\dots,{n-1}$.
\end{enumerate}
\end{definition}

\begin{proof}[Proof of Theorem \ref{thm:converse}]
We prove by contradiction. 
Assume that there exists $b \in B$ such that $([\omega_b],[\chi_b])$ is not $J$-nef, where $\omega_b$ and $\chi_b$ are restrictions to the fiber $X_b$ of $\omega_X$ and $\chi$ in the assumption of Theorem \ref{thm:converse}. 
 Then, there exists a $p$-dimensional subvariety $W_b$ of $X_b$ such that 
\begin{equation}\label{eq:V_b}
c_b = m \, \frac{\int_{X_b} \chi_b \wedge \omega_b^{m-1}}{\int_{X_b} \omega_b^m}
<
p \, \frac{\int_{W_b} \chi_b|_{W_b} \wedge \left(\omega_b|_{W_b}\right)^{p-1}}{\int_{W_b} \left(\omega_b|_{W_b}\right)^p}.
\end{equation}

By \eqref{eq:cexpan}, the constant $(m+n) \left(\int_X \chi \wedge \omega_k^{m+n-1}\right) /\left(\int_X \omega_k^{m+n}\right)$ converges to $c_b$ as $k \rightarrow \infty$.
By \eqref{eq:V_b}, this implies that the $p$-dimensional subvariety $W_b$ of $X$ satisfies 
$$(m+n) \, \frac{\int_X \chi \wedge \omega_k^{m+n-1}}{\int_X \omega_k^{m+n}}
<
p \, \frac{\int_{W_b} \chi|_{V_b} \wedge \left(\omega_X|_{W_b}\right)^{p-1}}{\int_{W_b} \left(\omega_X|_{W_b}\right)^p}$$
for all $k \gg 0$.
However, the pair $\left([\omega_k],[\chi]\right)$ is $J$-nef by assumption and this is a contradiction. 
In conclusion, the pair $([\omega_b],[\chi_b])$ on any fiber $X_b$ is $J$-nef.

Similarly, we prove $J$-nefness of the base by contradiction. 
Assume that the pair $\left([\omega_B],[\chi_B]\right)$ is not $J$-nef.
Then, there exists a $p$-dimensional subvariety $W$ of $B$ such that
\begin{equation}\label{eq:V}
n \, \frac{\int_{B} \chi_B \wedge \omega_B^{n-1}}{\int_B \omega_B^n}
<
p \, \frac{\int_W \chi_B \wedge \omega_B^{p-1}}{\int_W \omega_B^p}.
\end{equation}
By the assumption that 
$\Lambda_\mathcal{V} \chi = c_b$ or $\pi^*\left(\pi_B\left(\omega_X\right)_\mathcal{H}\right) = \left(\omega_X\right)_\mathcal{H}$, the constant $C_1$ of \eqref{eq:cexpan} vanishes. 
By the same calculation as \eqref{eq:cexpan}, we have
\begin{align*}
(m+p) \, \frac{\int_{\pi^{-1}W} \chi \wedge \omega_k^{m+p-1}}{\int_{\pi^{-1}W} \omega_k^{m+p}}
&= c_b + k^{-1} p \, \frac{\int_{\pi^{-1}W}\left(\omega_X\right)_\mathcal{V}^m \wedge \chi_\mathcal{H} \wedge \omega_B^{p-1} }{\int_{\pi^{-1}W} \left(\omega_X\right)_\mathcal{V}^m\wedge\omega_B^p} + O\left(k^{-2}\right)\\
&=c_b + k^{-1} p \, \frac{\int_W V_b \, \pi_B\left(\chi_\mathcal{H}\right) \wedge \omega_B^{p-1}}{\int_W V_b \, \omega_B^p} +O\left(k^{-2}\right)\\
&= c_b + k^{-1}p \, \frac{\int_W \chi_B \wedge \omega_B^{p-1}}{\int_W \omega_B^p} + O\left(k^{-2}\right).
\end{align*}
By \eqref{eq:V}, we have
$$ (m+n) \frac{\int_X \chi \wedge \omega_k^{m+n-1}}{\int_X \omega_k^{m+n}}
< 
(m+p)\frac{\int_{\pi^{-1}W} \chi \wedge \omega_k^{m+p-1}}{\int_{\pi^{-1}V} \omega_k^{m+p}}
$$
for any $k \gg 0$.
However, the pair $([\omega_k],[\chi])$ is $J$-nef by assumption and this is a contradiction. 
In conclusion, the pair $([\omega_B],[\pi_B(\chi_\mathcal{H})])$ is also $J$-nef.
\end{proof}

\begin{remark}\label{rem:collapsing}
If we assume the existence of a solution $\omega'_k = \omega_k+ \sqrt{-1}\partial\bar{\partial}\phi_k$ of the $J_\chi$-equation in the class $[\omega_k]$ for all $k \gg 0$,  
for small $t>0$ such that $\chi-t\omega'_k$ is a K{\"a}hler form, 
we have a solution $\omega'_k$ of the $J_{\chi-t\omega'_k}$-equation. 
By Theorem \ref{J-positivity}, we have 
$$\int_{\pi^{-1}W} c \, \omega_k^p - p \, \chi \wedge \omega_k^{p-1}
\ge
(m+n-p)t \, \int_{\pi^{-1}W} \omega_k^p$$
for all $p$-dimensional subvarieties $W$ of  $B$ with $p = 1,2,\dots,{n-1}$.
If a family of the solutions $\omega'_k$ has an order $O(k)$ as $k \rightarrow \infty$, we have $t = O\left(k^{-1}\right)$. Therefore, in the same assumptions of Theorem \ref{thm:converse}, we obtain the $J$-positivity of the pair $([\omega_B],[\pi_B(\chi_\mathcal{H})])$ by the above calculation. In the recent work \cite{GPT}, they proved the $L^\infty$ estimate for a solution of the family of Hessian equations with a certain structural condition. Our situation is not included in the class they studied. It is interesting to see if the method can apply to our case.
\end{remark}

\section{Deformed Hermitian-Yang-Mills equations}\label{sec:dHYM}
For a K{\"a}hler form $\chi$ and a closed real $(1,1)$-form $\omega$, we denote 
\begin{equation}\label{eq:theta}
    \theta_\chi(\omega):=\sum\mathrm{arccot}\, \lambda_i,
\end{equation}
where $\{\lambda_i\}$ are eigenvalues of $\omega\cdot\chi^{-1}$ and $\mathrm{arccot}(x)=\cot^{-1} x\in(0,\pi)$. Here $\cot\theta=1/\tan\theta$. 
Remark that $\theta_\chi$ depends on points, while $\theta$ in the dHYM equation \eqref{eq:dhym} is constant.
Denote by $\theta_\infty$ the constant (which is determined up to $2\pi$ by \eqref{eq:dhymconst}) of the dHYM$_{\chi_b}$ equation \eqref{eq:dhym} in the cohomology class $[\omega_X|_{X_b}]$.
Then, one will always be covered by at least one of the following four cases: the case
\begin{equation}\label{eq:ass}
    \theta_\infty\in(0,\pi)+2\pi\mathbb{Z}, 
\end{equation}
the case $\theta_\infty\in(\pi,2\pi)+2\pi\mathbb{Z}$, the case $\theta_\infty\in(-\pi/2,\pi/2)+2\pi\mathbb{Z}$, and the case $\theta_\infty\in(\pi/2,3\pi/2)+2\pi\mathbb{Z}$.
In this paper, we only discuss the case \eqref{eq:ass}. However, the other cases can be treated by the same manner (see Remark \ref{rem:other}).
The dHYM$_\chi$ equation \eqref{eq:dhym} is equivalent to
$$\cot\theta_\chi(\omega)=\cot\theta,$$
as long as $\theta_\chi(\omega)\notin\pi\mathbb{Z}$ and $\theta_\chi(\omega)-\theta\in(-\pi,\pi)+2\pi\mathbb{Z}$. 
In setup \ref{setup:dhym}, we see that this is the case by Lemma \ref{lem:dhymevalues} and \eqref{eq:ass}. 
In addition, we denote by $\theta$ the constant of the dHYM$_\chi$ equation in the cohomology class $[\omega_X+k\omega_B]$, although it depends on $k$.

\subsection{Preliminaries}\label{sec:dhympre}
Let $(X,\chi)$ be an $n$-dimensional compact K{\"a}hler manifold and $\omega$ a closed real $(1,1)$-form that satisfies $\theta_\chi(\omega)\notin\pi\mathbb{Z}$.
\begin{definition}\label{def:tilF-op}
    We define an operator $\tilde{F}_{\omega,\chi}:C^\infty(X)\to C^\infty(X)$ by
    \begin{align*}
        \tilde{F}_{\omega,\chi}(\phi)=&n\frac{\sqrt{-1}\partial\bar{\partial}\phi\wedge\left(\mathrm{Re}(\omega+\sqrt{-1}\chi)^{n-1}-\cot\theta_\chi(\omega)\mathrm{Im}(\omega+\sqrt{-1}\chi)^{n-1}\right)}{\mathrm{Im}(\omega+\sqrt{-1}\chi)^n}\\&\qquad-n\frac{\sqrt{-1}\partial\cot\theta_\chi(\omega)\wedge\bar{\partial}\phi\wedge\mathrm{Im}(\omega+\sqrt{-1}\chi)^{n-1}}{\mathrm{Im}(\omega+\sqrt{-1}\chi)^n}.
    \end{align*}
    Note that this is well defined since $\mathrm{Im}(\omega+\sqrt{-1}\chi)^n\neq0$ by $\theta_\chi(\omega)\notin\pi\mathbb{Z}$.
\end{definition}

\begin{remark}
    The operator $\tilde{F}_{\omega,\chi}$ becomes the linearization of the operator 
    $$\frac{\mathrm{Re}(\omega+\sqrt{-1}\chi)^{n}}{\mathrm{Im}(\omega+\sqrt{-1}\chi)^{n}}$$
    in $[\omega]$ at a solution of the dHYM$_\chi$ equation.
\end{remark}

\begin{lemma}\label{lem:dhymadjoint}
    The operator $\tilde{F}_{\omega,\chi}$ is a complex self-adjoint second order elliptic linear operator. Moreover, it satisfies
    \begin{align*}
        &\int_X \phi \tilde{F}_{\omega,\chi}(\psi)\mathrm{Im}(\omega+\sqrt{-1}\chi)^{n}\\
        =&\int_X n\sqrt{-1}\partial\phi\wedge \bar{\partial}\psi\wedge \left(\cot{\theta_\chi(\omega)}\mathrm{Im}(\omega+\sqrt{-1}\chi)^{n-1}-\mathrm{Re}(\omega+\sqrt{-1}\chi)^{n-1}\right).
    \end{align*}
    In particular,
    the subspace $\ker \tilde{F}_{\omega,\chi}$ of $C^\infty(X)$ consists of constant functions on $X$. 
\end{lemma}

\begin{proof}
Fix $p\in X$ and take a coordinate neighborhood around $p$ such that $\chi(p)=\sum\sqrt{-1}dz^i\wedge d\bar{z}^i$ and $\omega(p)=\sum\lambda_i\sqrt{-1}dz^i\wedge d\bar{z}^i$. Then, at $p$, we can see that
\begin{equation}\label{eq: posi}
\begin{split}
    &\left(\cos\theta_\chi(\omega)\mathrm{Im}(\omega+\sqrt{-1}\chi)^{n-1}-\sin\theta_\chi(\omega)\mathrm{Re}(\omega+\sqrt{-1}\chi)^{n-1}\right)\big/(n-1)! \\
    =&\sum_{i=1}^n \sin \left(\sum_{j\neq i} \mathrm{arccot}\,\lambda_j-\theta_\chi(\omega)\right)\left(\prod_{j\neq i}\sqrt{\lambda_j^2+1}\sqrt{-1}dz^j\wedge d\bar{z}^j\right)\\
    =&\sum_{i=1}^n \sin \left(-\mathrm{arccot}\,\lambda_i\right)\left(\prod_{j\neq i}\sqrt{\lambda_j^2+1}\sqrt{-1}dz^j\wedge d\bar{z}^j\right)<0,
\end{split}
\end{equation}
since $\mathrm{arccot}\, \lambda_i\in (0,\pi)$. Therefore, the operator $\tilde{F}_{\omega,\chi}$ is elliptic.
Moreover, using integration by parts, we obtain
\begin{align*}
    &\int_X \phi \tilde{F}_{\omega,\chi}(\psi)\mathrm{Im}(\omega+\sqrt{-1}\chi)^{\dim X}\\
    =&\int_X n\phi \sqrt{-1}\partial\bar{\partial}\psi\wedge\left(\mathrm{Re}(\omega+\sqrt{-1}\chi)^{n-1}-\cot\theta_\chi(\omega)\mathrm{Im}(\omega+\sqrt{-1}\chi)^{n-1}\right)\\ &\quad -n\phi\sqrt{-1}\partial\cot\theta_\chi(\omega)\wedge\bar{\partial}\psi\wedge\mathrm{Im}(\omega+\sqrt{-1}\chi)^{n-1}\\
    =&\int_Xn\sqrt{-1}\partial\phi\wedge \bar{\partial}\psi\wedge \left(\cot{\theta_\chi(\omega)}\mathrm{Im}(\omega+\sqrt{-1}\chi)^{n-1}-\mathrm{Re}(\omega+\sqrt{-1}\chi)^{n-1}\right).
\end{align*}
Thus, the operator $\tilde{F}_{\omega,\chi}$ is complex self-adjoint and $\ker\tilde{F}_{\omega,\chi}$ is precisely the set of constant functions by \eqref{eq: posi}.
\end{proof}

From now on we assume Setup \ref{setup:dhym}. In Setup \ref{setup:dhym}, the tangent space $TX$ splits as a smooth bundle 
\begin{equation}\label{eq:split}
    TX \cong \mathcal{V} \oplus \mathcal{H}, 
\end{equation}
where $\mathcal{V} = \mathop{\mathrm{ker}}d\pi $ denotes the vertical tangent bundle and $\mathcal{H}$ denotes the horizontal subbundle of $TX$ defined by 
$$ \mathcal{H}_x= \{ u \in T_x X \mid \chi (u, v)=0 \; \textrm{for all} \; v \in \mathcal{V}_x \}.$$ 
By this splitting, the form $\omega_X$ on $X$ is divided into the purely vertical component $(\omega_X)_\mathcal{V}$, the purely horizontal component $(\omega_X)_\mathcal{H}$ and the mixed component $(\omega_X)_m$.

\begin{lemma}\label{lem:dhymkahler}
    If $\omega_b$ is a solution of the dHYM$_{\chi_b}$ equation on each fiber $X_b$, then the form 
    $$\chi_B:=-\frac{\pi_*\Bigl(\mathrm{Re}(\omega_X+\sqrt{-1}\chi)^{m+1}-{c_0}\,\mathrm{Im}(\omega_X+\sqrt{-1}\chi)^{m+1}\Bigr)}{\pi_*\Bigl(\mathrm{Im}(\omega_X+\sqrt{-1}\chi)^m\Bigr)}$$
    is a K{\"a}hler form on $B$, where $c_0:=(\int_{X_b}\mathrm{Re}(\omega_b+\sqrt{-1}\chi_b)^m)/(\int_{X_b}\mathrm{Im}(\omega_b+\sqrt{-1}\chi_b)^m)$. 
\end{lemma}

\begin{proof}
The $J$-invariance and the closedness hold since $\omega_X$ and $\chi$ satisfy them. It remains to show the positivity.
Note that we have
\begin{align*}
    &\pi_*\Bigl(\mathrm{Re}(\omega_X+\sqrt{-1}\chi)^{m+1}-c_0\,\mathrm{Im}(\omega_X+\sqrt{-1}\chi)^{m+1}\Bigr)\\
    =&(m+1)\pi_*\Big(\mathrm{Re}(\omega_X+\sqrt{-1}\chi)_\mathcal{V}^m-c_0\,\mathrm{Im}(\omega_X+\sqrt{-1}\chi)_\mathcal{V}^m\Big)\wedge(\omega_X)_\mathcal{H}\\
    &-(m+1)\pi_*\Big(\mathrm{Im}(\omega_X+\sqrt{-1}\chi)_\mathcal{V}^m+c_0\,\mathrm{Re}(\omega_X+\sqrt{-1}\chi)_\mathcal{V}^m\Big)\wedge\chi_\mathcal{H}\\
    &+{m+1\choose2}\pi_*\Bigl(\mathrm{Re}(\omega_X+\sqrt{-1}\chi)_\mathcal{V}^{m-1}-c_0\,\mathrm{Im}(\omega_X+\sqrt{-1}\chi)_\mathcal{V}^{m-1}\Bigr)\wedge(\omega_X)_m^2.
\end{align*}
Since $\omega_b$ is a solution of the $\mathrm{dHYM}_{\chi_b}$ equation, the first term of the right-hand side vanishes, while the second term is negative since $\sin{\theta_\infty}$ and $\chi_\mathcal{H}$ is positive. Also, the same calculation as \eqref{eq: posi} shows that the $(m-1,m-1)$-form in the third term of the right-hand side is positive since $\sin\theta_\infty$ is positive. On the other hand, we can see that $i(Jv)i(v)(\omega_X)^2_m$ is a seminegative $(1,1)$-form, where $v\in \mathcal{H}$ and $i(\cdot)$ denotes the inner product. Indeed, let us denote 
$(\omega_X)_m=\sum (\omega_X)_{\alpha j}dx^\alpha\wedge dx^j$ and $v=\sum 
v^j\partial/\partial x^j$, 
where Greek letters (resp.\,Latin letters) represent the coordinates in the fiber (resp.\,horizontal) direction. Then, for $u\in\mathcal{V}$, by $J$-invariance of $\omega_X$, we have
\begin{align*}
    \left(i(Jv)i(v)(\omega_X)_m^2\right)(u,Ju)=&-(\omega_X)_{\alpha j}(\omega_X)_{\beta k} v^j (Jv)^k u^\alpha (Ju)^\beta\\
    =&-\left(\omega_X(u,v)\right)^2\le0
\end{align*}
which implies that $i(Jv)i(v)(\omega_X)^2_m$ is a seminegative $(1,1)$-form. Now as in the proof of Lemma \ref{lem:Kahler}, fix a point $b \in B$ and a vector $u \in T_bB$ and take a local trivialization $\psi: X_U \cong Y \times U$ on a neighborhood $U$ of $b$, such that the splitting $TX \cong \mathcal{V} \oplus \mathcal{H}$ coincides with $TY \oplus TU$ on a fiber $X_b$ via $\psi_*$. Then, for $u\in T_bB$, we see that
$$-\Bigl(\mathrm{Re}(\omega_X+\sqrt{-1}\chi)^{m+1}-c_0\,\mathrm{Im}(\omega_X+\sqrt{-1}\chi)^{m+1}\Bigr)(\psi^{-1}_*u,J\psi^{-1}_*u)$$
is a volume form on $X_b$. Thus, the same calculation as in the proof of Lemma \ref{lem:Kahler} shows that $\chi_B$ is positive.
\end{proof}

For a function $f$ on $X$, we define a function $f_B$ by
\begin{equation}\label{eq:pitilde}
    \tilde{\pi}_B(f)(b)=V_b^{-1}\int_{X_b}f|_{X_b}\, \mathrm{Im}(\omega_b+\sqrt{-1}\chi_b)^m,
\end{equation}
where $V_b:=\int_X\mathrm{Im}(\omega_b+\sqrt{-1}\chi_b)^m$.
Denote the subspace $\ker\bigl(\tilde{\pi}_B:C^\infty(X)\to C^\infty(B)\bigr)$ by $C^\infty_0(X)$. In particular, we have a decomposition $f=f_B+f_\mathcal{V}$, where $f_B=\pi^*\tilde{\pi}_B(f)\in C^\infty(B)$ and $f_\mathcal{V}=f-f_B\in C_0^\infty(X)$.

\subsection{Approximate solutions}
In this subsection, as in the $J$-equation case, we construct a family of approximate solutions:
\begin{proposition}\label{prop:dhymapp}
    For any $r\in\mathbb{Z}_{\ge0}$, there exist $\{\phi_{i,B}\}^r_{i=0}\subset C^\infty(B)$ and $\{\phi_{i,\mathcal{V}}\}^r_{i=0}\subset C^\infty_0(X)$ such that
    \begin{equation}\label{eq:dhymapp}
        \omega_{k, r} = \omega_k + \sqrt{-1} \partial \bar{\partial} \left( \sum_{i=0}^r \phi_{i, B} k^{2-i} + \sum_{i=0}^r \phi_{i, \mathcal{V}} k^{-i} \right)
    \end{equation}
    satisfies 
    $$\frac{\mathrm{Re}(\omega_{k,r}+\sqrt{-1}\chi)^{m+n}}{\mathrm{Im}(\omega_{k,r}+\sqrt{-1}\chi)^{m+n}}=\sum_{i=0}^r k^{-i} {c_i} + O \left(k^{-r-1} \right), $$
    where $ {c_i} $ are constants.
\end{proposition}

We prove this proposition by the same arguments as in the $J$-equation case.

\begin{lemma}\label{lem:dhymO(1)-term}
Suppose that $\omega_b$ is a solution of the dHYM$_{\chi_b}$ equation on each fiber $X_b$ and $\omega_B$ is a solution of the $J_{\chi_B}$-equation. Then, we have
\begin{align*}
    \frac{\mathrm{Re}(\omega_k+\sqrt{-1}\chi)^{m+n}}{\mathrm{Im}(\omega_k+\sqrt{-1}\chi)^{m+n}}
    =c_0+k^{-1}\left(c_1+(f_1)_{\mathcal{V}}\right)+O(k^{-2}),
\end{align*}
where $c_0$ and $c_1$ are constants and $(f_1)_{\mathcal{V}}$ is a function in $C^\infty_0(X)$.
\end{lemma}

\begin{proof}
    We have
    \begin{align*}
        &\frac{\mathrm{Re}(\omega_X+\sqrt{-1}\chi+k\omega_B)^{m+n}}{\mathrm{Im}(\omega_X+\sqrt{-1}\chi+k\omega_B)^{m+n}}\\
        =&\frac{\mathrm{Re}(\omega_X+\sqrt{-1}\chi)^m\wedge\omega_B^n}{\mathrm{Im}(\omega_X+\sqrt{-1}\chi)^m\wedge\omega_B^n}+k^{-1}\frac{n}{m+1}\frac{\mathrm{Re}(\omega_X+\sqrt{-1}\chi)^{m+1}\wedge\omega_B^{n-1}}{\mathrm{Im}(\omega_X+\sqrt{-1}\chi)^m\wedge\omega_B^n}\\
        &-k^{-1}\frac{n}{m+1}\frac{\mathrm{Re}(\omega_X+\sqrt{-1}\chi)^m\wedge\omega_B^n}{\mathrm{Im}(\omega_X+\sqrt{-1}\chi)^m\wedge\omega_B^n}\frac{\mathrm{Im}(\omega_X+\sqrt{-1}\chi)^{m+1}\wedge\omega_B^{n-1}}{\mathrm{Im}(\omega_X+\sqrt{-1}\chi)^m\wedge\omega_B^n}+O(k^{-2})\\
        =&c_0\\
        &+k^{-1}\frac{n}{m+1}\left(\frac{\mathrm{Re}(\omega_X+\sqrt{-1}\chi)^{m+1}\wedge\omega^{n-1}_B-c_0\,\mathrm{Im}(\omega_X+\sqrt{-1}\chi)^{m+1}\wedge\omega_B^{n-1}}{\mathrm{Im}(\omega_X+\sqrt{-1}\chi)^m\wedge\omega_B^n}\right)\\
        &+O(k^{-2}).
    \end{align*}
    Now, the $C^\infty(B)$-part in the decomposition by \eqref{eq:pitilde} of the term in the second last line is constant by the assumption that $\omega_B$ is a solution of the $J_{\chi_B}$-equation. Thus, by defining
    \begin{equation}\label{eq:dhymc_1}
        c_1=-\frac{n}{m+1}\frac{\chi_B\wedge\omega_B^{n-1}}{\omega_B^n},
    \end{equation}
    we obtain the assertion.
\end{proof}

\begin{definition}
For $\phi \in C^\infty(X)$, we define the operator $\tilde{F}_{ \left( \omega_X \right)_\mathcal{V}, \chi_\mathcal{V}}$ by 
$$\Bigl( \tilde{F}_{ \left( \omega_X \right)_\mathcal{V}, \chi_\mathcal{V}} (\phi) \Bigr) \Big\vert_{X_b} = \tilde{F}_{ \omega_b, \chi_b} \left( \phi\vert_{X_b} \right).$$
Note that $\tilde{F}_{ \left( \omega_X \right)_\mathcal{V}, \chi_\mathcal{V}} (\phi)\in C^\infty(X)$ since the operator $\tilde{F}_{\omega_b,\chi_b}$ depends smoothly on $b$ by Definition \ref{def:tilF-op}.
\end{definition}

\begin{lemma}
 For $\phi \in C^\infty(X)$, we have
$$\frac{d}{dt}\Big|_{t = 0}\frac{\mathrm{Re}(\omega_{k,r,t\phi}+\sqrt{-1}\chi)^{m+n}}{\mathrm{Im}(\omega_{k,r,t\phi}+\sqrt{-1}\chi)^{m+n}}
= \tilde{F}_{ \left( \omega_X \right)_\mathcal{V}, \chi_\mathcal{V}} \left(\phi \right) + k^{-1}D_1(\phi) + k^{-2}D_2(\phi)+O(k^{-3}),$$
where the operators $D_1$ and $D_2$ satisfy $D_1(\phi)=0$ and $\tilde{\pi}_B\bigl(D_2(\phi)\bigr)=-F_{\omega_B,\chi_B}(\phi)/(m+1)$ for $\phi\in C^\infty(B)$. Here, the operator $F_{\omega_B,\chi_B}$ is given by Definition \ref{def:F-op} and we suppress pullbacks via $\pi$,
so for a function $\phi\in C^\infty(B)$ its pullback to $X$ is also denoted by $\phi$.
\end{lemma}

\begin{proof}
    We have
    \begin{align*}
        &\frac{d}{dt}\Big|_{t = 0}\frac{\mathrm{Re}(\omega_{k,r,t\phi}+\sqrt{-1}\chi)^{m+n}}{\mathrm{Im}(\omega_{k,r,t\phi}+\sqrt{-1}\chi)^{m+n}}\\
        =&(m+n)\frac{\mathrm{Re}(\omega_{k,r}+\sqrt{-1}\chi)^{m+n-1}\wedge\sqrt{-1}\partial\bar{\partial}\phi}{\mathrm{Im}(\omega_{k,r}+\sqrt{-1}\chi)^{m+n}}\\
        &-(m+n)\frac{\mathrm{Re}(\omega_{k,r}+\sqrt{-1}\chi)^{m+n}}{\mathrm{Im}(\omega_{k,r}+\sqrt{-1}\chi)^{m+n}}\frac{\mathrm{Im}(\omega_{k,r}+\sqrt{-1}\chi)^{m+n-1}\wedge\sqrt{-1}\partial\bar{\partial}\phi}{\mathrm{Im}(\omega_{k,r}+\sqrt{-1}\chi)^{m+n}}\\
        =&m\frac{\mathrm{Re}((\omega_X)_\mathcal{V}+\sqrt{-1}\chi_\mathcal{V})^{m-1}\wedge(\sqrt{-1}\partial\bar{\partial}\phi)_\mathcal{V}}{\mathrm{Im}((\omega_X)_\mathcal{V}+\sqrt{-1}\chi_\mathcal{V})^m}\\
        &-m\frac{\mathrm{Re}((\omega_X)_\mathcal{V}+\sqrt{-1}\chi_\mathcal{V})^m}{\mathrm{Im}((\omega_X)_\mathcal
        V+\sqrt{-1}\chi_\mathcal{V})^m}\frac{\mathrm{Im}((\omega_X)_\mathcal{V}+\sqrt{-1}\chi_\mathcal{V})^{m-1}\wedge(\sqrt{-1}\partial\bar{\partial}\phi)_\mathcal{V}}{\mathrm{Im}((\omega_X)_\mathcal{V}+\sqrt{-1}\chi_\mathcal{V})^m}\\
        &+k^{-1}D_1(\phi)+k^{-2}D_2(\phi)+O(k^{-3})\\
        =&\tilde{F}_{(\omega_X)_\mathcal{V},\chi_\mathcal{V}}(\phi)+k^{-1}D_1(\phi)+k^{-2}D_2(\phi)+O(k^{-3}),
    \end{align*}
    where the first equation follows since $\sqrt{-1}\partial\bar{\partial}\phi$ is real.
    For $\phi\in C^\infty(B)$, noting that $\sqrt{-1}\partial\bar{\partial}\phi$ has only a horizontal part, 
    we have
    \begin{align*}
        D_1(\phi)
        =&n\frac{\mathrm{Re}((\omega_X)_\mathcal{V}+\sqrt{-1}\chi_\mathcal{V})^m\wedge\omega_B^{n-1}\wedge\sqrt{-1}\partial\bar{\partial}\phi}{\mathrm{Im}((\omega_X)_\mathcal{V}+\sqrt{-1}\chi_\mathcal{V})^m\wedge\omega_B^n}\\
        &-n\frac{\mathrm{Re}((\omega_X)_\mathcal{V}+\sqrt{-1}\chi_\mathcal{V})^m}{\mathrm{Im}((\omega_X)_\mathcal{V}+\sqrt{-1}\chi_\mathcal{V})^m}\frac{\mathrm{Im}((\omega_X)_\mathcal{V}+\sqrt{-1}\chi_\mathcal{V})^m\wedge\omega_B^{n-1}\wedge\sqrt{-1}\partial\bar{\partial}\phi}{\mathrm{Im}((\omega_X)_\mathcal{V}+\sqrt{-1}\chi_\mathcal{V})^m\wedge\omega_B^n}\\
        =&0
    \end{align*}
    Moreover, we have
    \begin{align*}
        &\tilde{\pi}_BD_2(\phi)\\
        =&\tilde{\pi}_B\Biggl(\frac{n(n-1)}{m+1}\frac{\mathrm{Re}(\omega_X+\sqrt{-1}\partial\bar{\partial}\phi_{2,B}+\sqrt{-1}\chi)^{m+1}\wedge\omega_B^{n-2}\wedge\sqrt{-1}\partial\bar{\partial}\phi}{\mathrm{Im}((\omega_X)_\mathcal{V}+\sqrt{-1}\chi_\mathcal{V})^m\wedge\omega_B^n}\\
        &\qquad-\frac{n(n-1)}{m+1}\frac{\mathrm{Re}((\omega_X)_\mathcal{V}+\sqrt{-1}\chi_\mathcal{V})^m}{\mathrm{Im}((\omega_X)_\mathcal{V}+\sqrt{-1}\chi_\mathcal{V})^m}\frac{\mathrm{Im}(\omega_X+\sqrt{-1}\partial\bar{
        \partial}\phi_{2,B}+\sqrt{-1}\chi)^{m+1}\wedge\omega_B^{n-2}\wedge\sqrt{-1}\partial\bar{\partial}\phi}{\mathrm{Im}((\omega_X)_\mathcal{V}+\sqrt{-1}\chi_\mathcal{V})^m\wedge\omega_B^n}\Biggr)\\
        &+\tilde{\pi}_B\Biggl(nm\frac{\mathrm{Re}((\omega_X)_\mathcal{V}+\sqrt{-1}\chi_\mathcal{V})^{m-1}\wedge\sqrt{-1}\partial\bar{\partial}\phi_{1,\mathcal{V}}\wedge\omega_B^{n-1}\wedge\sqrt{-1}\partial\bar{\partial}\phi}{\mathrm{Im}((\omega_X)_\mathcal{V}+\sqrt{-1}\chi_\mathcal{V})^m\wedge\omega_B^n}\\
        &\qquad-nm\frac{\mathrm{Re}((\omega_X)_\mathcal{V}+\sqrt{-1}\chi_\mathcal{V})^m}{\mathrm{Im}((\omega_X)_\mathcal{V}+\sqrt{-1}\chi_\mathcal{V})^m}\frac{\mathrm{Im}((\omega_X)_\mathcal{V}+\sqrt{-1}\chi_\mathcal{V})^{m-1}\wedge
        \sqrt{-1}\partial\bar{\partial}\phi_{1,\mathcal{V}}\wedge\omega_B^{n-1}\wedge\sqrt{-1}\partial\bar{\partial}\phi}{\mathrm{Im}((\omega_X)_\mathcal{V}+\sqrt{-1}\chi_\mathcal{V})^m\wedge\omega_B^n}\Biggr)\\
        &-\tilde{\pi}_B\Biggl(c_1n\frac{\omega_B^{n-1}\wedge\sqrt{-1}\partial\bar{\partial}\phi}{\omega_B^n}\Biggr)\\
        =-&\frac{n(n-1)}{m+1}\frac{\chi_B\wedge\omega_B^{n-2}\wedge\sqrt{-1}\partial\bar{\partial}\phi}{\omega_B^n}-c_1n\frac{\omega_B^{n-1}\wedge\sqrt{-1}\partial\bar{\partial}\phi}{\omega_B^n}\\
        =-&F_{\omega_B,\chi_B}(\phi)/(m+1).
    \end{align*}
    Thus, we confirm the assertion.
\end{proof}

Combining the lemmas we obtained, the same arguments as in the proof of Proposition \ref{prop:app} give a proof of Proposition \ref{prop:dhymapp}.

\subsection{First proof of Theorem \ref{thm:dhym_main} under the supercritical phase condition}\label{sec:dhymfirst}
A solution $\omega$ of the dHYM$_\chi$ equation is equivalent to a solution of  
\begin{equation}\label{eq:dhymlift}
    \theta_\chi(\omega) = \hat{\theta},
\end{equation}
where $\hat{\theta} \in \mathbb{R}$ is a constant satisfying \eqref{eq:dhymconst}.  
The constant $\hat{\theta}$ is called the \emph{lifted phase}.  
Note that the solution $\omega$ determines $\hat{\theta}$, whereas condition \eqref{eq:dhymconst} only determines the phase $\theta$ modulo $2\pi$.  
According to \cite[Definition~2.5]{CXY}, if there exists a form $\omega$ such that $\operatorname{osc}\theta_\chi(\omega) < \pi$,  
then the lifted phase $\hat{\theta}$ can be uniquely defined in the interval $[\inf \theta_\chi(\omega),\, \sup \theta_\chi(\omega)]$.  
Under the assumption of Theorem~\ref{thm:dhym_main}, Lemma~\ref{lem:dhymevalues} shows that this condition is satisfied.  

The \emph{supercritical phase condition} is defined by requiring that $\hat{\theta} \in (0,\pi)$.  
In our setting, Lemma~\ref{lem:dhymevalues} further implies that the supercritical condition for $k \gg 0$  
is equivalent to $\hat{\theta}_\infty \in (0,\pi)$,  
where $\hat{\theta}_\infty$ denotes the lifted phase of the dHYM equation on the fiber determined by the solution.  

The solvability of the supercritical dHYM equation is characterized by the following theorem:
\begin{theorem}[{\cite[Theorem 1.2]{CJY}}]\label{thm:CJY}
    Let $(X,\chi)$ be an $n$-dimensional compact K{\"a}hler manifold and $\omega$ a closed real $(1,1)$-form. 
    The following are equivalent:
    \begin{enumerate}[font=\normalfont]
        \item There exists a solution $\omega'\in[\omega]$ of the dHYM$_\chi$ equation \eqref{eq:dhymlift} with $\hat{\theta}\in (0,\pi)$.
        \item There exists a supercritical subsolution $\omega'\in[\omega]$ of the dHYM$_\chi$ equation in the sense of \cite{CJY}, that is, the closed real $(1,1)$-form $\omega'\in[\omega]$ satisfies $\theta_\chi(\omega')\in(0,\pi)$ and $\mathrm{Re}(\omega'+\sqrt{-1}\chi)^{n-1}-\cot\theta \, \mathrm{Im}(\omega'+\sqrt{-1}\chi)^{n-1}>0$.        
    \end{enumerate}
\end{theorem}

\begin{proof}[Proof of Theorem \ref{thm:dhym_main} under the supercritical phase condition]
    By Theorem \ref{thm:CJY}, it suffices to find a supercritical subsolution in $[\omega_k]$ for $k\gg0$.
    By a simple calculation, we can expand $\cot\theta=\sum k^{-i}c_i$ as 
    \begin{equation}\label{eq:dhymcexpan}
        \begin{split}
        &\cot\theta=\frac{\int_X\mathrm{Re(\omega_k+\sqrt{-1}\chi)^{m+n}}}{\int_X\mathrm{Im}(\omega_k+\sqrt{-1}\chi)^{m+n}}\\
        =&\frac{\int_X\mathrm{Re}(\omega_X+\sqrt{-1}\chi)^m\wedge\omega_B^n}{\int_X\mathrm{Im}(\omega_X+\sqrt{-1}\chi)^m\wedge\omega_B^n}+k^{-1}\frac{n}{m+1}\frac{\int_X\mathrm{Re}(\omega_X+\sqrt{-1}\chi)^{m+1}\wedge\omega_B^{n-1}}{\int_X\mathrm{Im}(\omega_X+\sqrt{-1}\chi)^m\wedge\omega_B^n}\\
        &+k^{-1}\frac{n}{m+1}\frac{\int_X\mathrm{Re}(\omega_X+\sqrt{-1}\chi)^m\wedge\omega_B^n}{\int_X\mathrm{Im}(\omega_X+\sqrt{-1}\chi)^m\wedge\omega_B^n}\frac{\int_X\mathrm{Im}(\omega_X+\sqrt{-1}\chi)^{m+1}\wedge\omega_B^{n-1}}{\int_X\mathrm{Im}(\omega_X+\sqrt{-1}\chi)^m\wedge\omega_B^n}\\
        &+O(k^{-2}).
        \end{split}
    \end{equation}
    The coefficient $c_0$ of the $k^{0}$-order term of \eqref{eq:dhymcexpan} is the constant of the dHYM$_{\chi_b}$ equation. Since $\omega_b$ is a solution of the dHYM$_{\chi_b}$ equation on each fiber $X_b$, we see that the coefficient $c_1$ of the $k^{-1}$-order term of \eqref{eq:dhymcexpan} is the same as \eqref{eq:dhymc_1}.
    Fix a point $p$ in $X$. Let $\alpha$ be a non-zero semipositive form at $p$. We first calculate with $\omega_k$. At $p$, we have
    \begin{align*}
        &\left(\mathrm{Re}(\omega_k+\sqrt{-1}\chi)^{m+n-1}-\cot\theta\, \mathrm{Im}(\omega_k+\sqrt{-1}\chi)^{m+n-1}\right)\wedge\alpha\\
        =&k^n {{m+n-1}\choose n} \big(\mathrm{Rm}(\omega_X+\sqrt{-1}\chi)^{m-1}-c_0\mathrm{Im}(\omega_X+\sqrt{-1}\chi)^{m-1}\big)\wedge\omega_B^n\wedge\alpha\notag\\
        &+k^{n-1}{m+n-1\choose n-1} \big(\mathrm{Rm}(\omega_X+\sqrt{-1}\chi)^m -c_0\mathrm{Im}(\omega_X+\sqrt{-1}\chi)^m\big)\wedge\omega_B^{n-1}\wedge\alpha\notag\\
        &-k^{n-1}{m+n-1\choose n}c_1\mathrm{Im}(\omega_X+\sqrt{-1}\chi)^{m-1}\wedge\omega_B^n\wedge\alpha\notag\\
        &+k^{n-2}{{m+n-1}\choose n-2} \big(\mathrm{Re}(\omega_X+\sqrt{-1}\chi)^{m+1}-c_0\mathrm{Im}(\omega_X+\sqrt{-1}\chi)^{m+1}\big)\wedge\omega_B^{n-2}\wedge\alpha\notag\\
        &-k^{n-2} {{m+n-1}\choose n-1} c_1\mathrm{Im}(\omega_X+\sqrt{-1}\chi)^m\wedge\omega_B^{n-1}\wedge\alpha\notag\\
        &-k^{n-2}{{m+n-1}\choose n} c_2\mathrm{Im}(\omega_X+\sqrt{-1}\chi)^{m-1}\wedge\omega_B^{n}\wedge\alpha+O(k^{n-3})\notag.
    \end{align*}
    We first consider the $k^n$-order term. Since $\omega_{\pi(p)}$ is a solution of the dHYM$_{\chi_{\pi(p)}}$ equation on each fiber, it is a supercritical subsolution of the dHYM$_{\chi_{\pi(p)}}$ equation on each fiber. By the assumption that $\alpha_{\mathcal{V}}$ is a semipositive form at $p$, we see that the $k^n$-order term is positive if $\alpha_{\mathcal{V}}$ is non-zero. Let us perturb $\omega_{k}$ to $\omega_{k,1}:=\omega_k+\sqrt{-1}\partial\bar{\partial}k^{-1}\phi_{1,\mathcal{V}}$. Then, since this perturbation only contains $O(k^{-1})$-term, the leading term is unchanged and hence still positive if $\alpha_{\mathcal{V}}$ is non-zero. 
    
    Now let us assume $\alpha_{\mathcal{V}}$ is zero. 
    Then, as in the proof of Theorem \ref{thm:main} in Subsection \ref{sec:firstproof}, the mixed term $\alpha_m$ is also zero.  
    In the above calculation, the first term of the $k^{n-1}$-order terms vanishes since $\alpha_{\mathcal{V}}$ is zero and $\omega_{\pi(p)}$ is a solution of the dHYM$_{\chi_{\pi(p)}}$ equation. The second term of the $k^{n-1}$-order terms and the third term of the $k^{n-2}$-order terms also vanish since $\alpha_{\mathcal{V}}$ is zero. 
    When replacing $\omega_{k}$ with $\omega_{k,1}:=\omega_k+\sqrt{-1}\partial\bar{\partial}k^{-1}\phi_{1,\mathcal{V}}$, there is another $k^{n-1}$-order term coming from the combination of the $k^n$-order term in the above calculation and $k^{-1}\sqrt{-1}\bar{\partial}\partial\phi_{1,\mathcal{V}}$. However, this term still vanishes since $\alpha_\mathcal{V}=0$.
    By the assumption that $\alpha=\alpha_{\mathcal{H}}$, we see that the $k^{n-2}$-order term becomes
    \begin{align}\label{eq:dhympush}
        &k^{n-2}{{m+n-1}\choose n-2} \big(\mathrm{Re}(\omega_X+\sqrt{-1}\chi)^{m+1}-c_0\mathrm{Im}(\omega_X+\sqrt{-1}\chi)^{m+1}\big)\wedge\omega_B^{n-2}\wedge\alpha\\
        -&k^{n-2} {{m+n-1}\choose n-1} c_1\mathrm{Im}(\omega_X+\sqrt{-1}\chi)^m\wedge\omega_B^{n-1}\wedge\alpha\notag\\
        +&k^{n-2}{m+n-1\choose n-1} \big(\mathrm{Rm}(\omega_X+\sqrt{-1}\chi)^{m-1} -c_0\mathrm{Im}(\omega_X+\sqrt{-1}\chi)^{m-1}\big)\wedge\sqrt{-1}\partial\bar{\partial}\phi_{1,\mathcal{V}}\wedge\omega_B^{n-1}\wedge\alpha\notag.
    \end{align}
    The function $\phi_{1,\mathcal{V}}$ satisfies $\tilde{F}_{(\omega_{X})_\mathcal{V},\chi_{\mathcal{V}}}(\phi_{1,\mathcal{V}})=-(f_1)_{\mathcal{V}}$, which in other words means that the last term of \eqref{eq:dhympush} equals 
    $$-k^{n-2}{m+n-1\choose n-1}(f_1)_{\mathcal{V}}\, \mathrm{Im}(\omega_X+\sqrt{-1}\chi)^m\wedge\omega_B^{n-1}\wedge\alpha.$$ 
    Let us define a $(1,1)$-form $\eta_\mathcal{H}$ by $$\eta_\mathcal{H}\wedge\mathrm{Im}(\omega_X+\sqrt{-1}\chi)_{\mathcal{V}}^m=-\big(\mathrm{Re}(\omega_X+\sqrt{-1}\chi)^{m+1}-c_0\mathrm{Im}(\omega_X+\sqrt{-1}\chi)^{m+1}\big).$$
    Then, the arguments in the proof of Lemma \ref{lem:dhymkahler} implies that $\eta_\mathcal{H}$ is a positive $(1,1)$-form.
    Take local holomorphic coordinates of $\mathcal{H}$ around $p$ in which $\omega_B=\sum\sqrt{-1}dz^i\wedge d\bar{z}^i$ and $\eta_{\mathcal{H}}=\sum\lambda_i\sqrt{-1}dz^i\wedge d\bar{z}^i$ at $p$. Then, we see that
    \begin{align*}
        &-(n-1)\eta_{\mathcal{H}}\wedge\omega_B^{n-2}-(m+1)\big(c_1+(f_1)_{\mathcal{V}}\big)\omega_B^{n-1}\\
        =&-(n-1)\eta_{\mathcal{H}}\wedge\omega_B^{n-2}+(\Lambda_{\omega_B}\eta_\mathcal{H})\omega_B^{n-1}\\
        =&\sum_i\left(\lambda_i\prod_{j\neq i}\sqrt{-1}dz^j\wedge d\bar{z}^j\right)>0,
    \end{align*}
    where the first equality follows from the definition of $(f_1)_{\mathcal{V}}$ (see the proof of Lemma \ref{lem:dhymO(1)-term}), and the last positivity follows from that $\eta_{\mathcal{H}}$ is a positive form. Therefore, in either case where $\alpha_{\mathcal{V}}$ are zero or not, we have proved that $(\mathrm{Re}(\omega_{k,1}+\sqrt{-1}\chi)^{m+n-1}-\cot\theta\, \mathrm{Im}(\omega_{k,1}+\sqrt{-1}\chi)^{m+n-1})\wedge\alpha$ is positive for $k\gg0$. Since the arguments above do not depend on $p$ or $\alpha$, we have proved that the form $\mathrm{Re}(\omega_{k,1}+\sqrt{-1}\chi)^{m+n-1}-\cot\theta\, \mathrm{Im}(\omega_{k,1}+\sqrt{-1}\chi)^{m+n-1}$ is positive for $k\gg0$. Lastly, by Lemma \ref{lem:dhymevalues} below and the supercritical phase condition $\hat{\theta}_\infty\in(0,\pi)$, we see that $\omega_{k,1}$ satisfies $\theta_\chi(\omega_{k,1})\in(0,\pi)$ for $k\gg0$.
\end{proof}

\subsection{Proof of Theorem \ref{thm:dhym_main}}

\begin{lemma}\label{lem:dhymevalues}
    The eigenvalues $\{\lambda_i(k,r)\}$ of $\omega_{k,r}\cdot\chi^{-1}$ either converge to those of $(\omega_X)_\mathcal{V}\cdot\chi^{-1}_\mathcal{V}$ or diverge in the order $O(k)$.
\end{lemma}

\begin{proof}
    Denote $\Bigl( \displaystyle\sum_{i=0}^r k^{-i} \phi_{i, \mathcal{V}} \Bigr)$ by $\left( \phi_\mathcal{V} \right)_{k,r}$ and 
    $\Bigl( \displaystyle\sum_{i=0}^r k^{-i} \phi_{i, B} \Bigr)$ by $\left( \phi_B \right)_{k,r}$ for simplicity. Fix a point $p\in X$. Take a coordinate around $p$ independent of $k$ on which the form $\omega_{k,r}$ is represented as the matrix
    \begin{align*}
    &{\omega_{k,r}} \\ 
    = &\begin{bmatrix}
    \left( \omega_X \right)_\mathcal{V} + \Bigl(\sqrt{-1} \partial \bar{\partial} \left( \phi_\mathcal{V} \right)_{k,r}\Bigr)_\mathcal{V}
    & \left(\omega_X\right)_m+\Bigl(\sqrt{-1} \partial \bar{\partial} \left( \phi_\mathcal{V} \right)_{k,r}\Bigr)_m  \\
    \left(\omega_X\right)_m+\Bigl(\sqrt{-1} \partial \bar{\partial} \left( \phi_\mathcal{V} \right)_{k,r}\Bigr)_m
    & \left( \omega_k \right)_\mathcal{H}+ \Bigl(\sqrt{-1} \partial \bar{\partial} \bigl(\left( \phi_\mathcal{V} \right)_{k,r} + \left( \phi_B \right)_{k,r} \bigr)\Bigr)_\mathcal{H}
\end{bmatrix}.
\end{align*}  
Then, we have
$$\begin{bmatrix}
    I & 0 \\ 0 & \frac{1}{\sqrt{k}}I
\end{bmatrix}
\omega_{k,r}
\begin{bmatrix}
    I & 0 \\ 0 & \frac{1}{\sqrt{k}}I
\end{bmatrix} \rightarrow 
\begin{bmatrix}
    (\omega_X)_\mathcal{V} & 0 \\ 0 & \omega_B
\end{bmatrix} \text{as $k \rightarrow \infty$}.$$
This convergence implies that some eigenvalues of $\omega_{k,r}$ on this coordinate converge to those of $(\omega_X)_\mathcal{V}$ and the others converge to those of $\omega_B$ multiplied by $k$.
\end{proof}

To do the same calculations as in the $J$-equation case, we define a reference K{\"a}hler form $\chi_k:=\chi+k\pi^*\omega_B$, where $\omega_B$ is a K{\"a}hler form on the base $B$.

\begin{lemma}
    Denote the first non-zero eigenvalue of $-\tilde{F}_{\omega_{k,r},\chi}$ by $\tilde{\lambda}_1$. Then, there exists a constant $C$ independent of $k$ such that $\tilde{\lambda}_1\ge Ck^{-2}$ for any $k$.
\end{lemma}

\begin{proof}
    Let $\phi_1$ be an eigenfunction corresponding to $\tilde{\lambda}_1$. Using Lemma \ref{lem:dhymadjoint}, we have
    \begin{align*}
        \tilde{\lambda}_1
        &=-\left(\int_X\phi_1\, \tilde{F}_{\omega_{k,r},\chi}(\phi_1)\, \mathrm{Im}(\omega_{k,r}+\sqrt{-1}\chi)^{m+n}\right)\Big/\left(\int_X\phi_1^2\, \mathrm{Im}(\omega_{k,r}+\sqrt{-1}\chi)^{m+n}\right)\\
        &=\left(\int_X (m+n)\sqrt{-1}\partial\phi_1\wedge \bar{\partial}\phi_1\right.\\
        &\qquad\wedge \left(\mathrm{Re}(\omega_{k,r}+\sqrt{-1}\chi)^{m+n-1}-\cot{\theta_\chi(\omega_{k,r})}\mathrm{Im}(\omega_{k,r}+\sqrt{-1}\chi)^{m+n-1}\right)\bigg)\\
        &\qquad\times\left(\int_X\phi_1^2\, \mathrm{Im}(\omega_{k,r}+\sqrt{-1}\chi)^{m+n}\right)^{-1}.
    \end{align*}
    Let us denote the eigenvalues of $\omega_{k,r}\cdot\chi^{-1}$ by $\{\lambda_i(k,r)\}_i$. Fix a point $p\in X$. Take a coordinate around $p$ on which $\chi=\sum_i\sqrt{-1}dz^i\wedge d\bar{z}^i$ and $\omega_{k,r}=\sum_i \lambda_i(k,r)dz^i\wedge d\bar{z}^i$ at $p$.
    Then, at $p$, we have
    \begin{align*}
        &\left(\mathrm{Re}(\omega_{k,r}+\sqrt{-1}\chi)^{m+n-1}-\cot{\theta_\chi(\omega_{k,r})}\mathrm{Im}(\omega_{k,r}+\sqrt{-1}\chi)^{m+n-1}\right)\\
        =&\sum_i\frac{1}{\sin\theta_\chi(\omega_{k,r})}\big(\sin(\mathrm{arccot}\,\lambda_i(k,r))\big)\left(\prod_{j\neq i}\sqrt{\lambda_j(k,r)^2+1}\, \sqrt{-1}dz^j\wedge d\bar{z}^j\right)
    \end{align*}
    By Lemma \ref{lem:dhymevalues}, we have either $\mathrm{arccot}\, \lambda_i(k,r)=O(1)$ or $\mathrm{arccot}\, \lambda_i(k,r)=Ck^{-1}+O(k^{-2})$ for some constant $C>0$. Indeed, the latter follows from 
    \begin{align*}
        \mathrm{arccot}\, \lambda_i(k,r)=\mathrm{arccot}\, (k(C+O(k^{-1})))=\frac{1}{C}k^{-1}+O(k^{-2}),
    \end{align*}
    where we used that $\mathrm{arccot}'(x)=-1/(1+x^2)$ in the second equality. Then, we have
    \begin{align*}
        &\sqrt{-1}\partial\phi_1\wedge\bar{\partial}\phi_1\wedge\left(\mathrm{Re}(\omega_{k,r}+\sqrt{-1}\chi)^{m+n-1}-\cot{\theta_\chi(\omega_{k,r})}\mathrm{Im}(\omega_{k,r}+\sqrt{-1}\chi)^{m+n-1}\right)\\
        \ge&Ck^{-1}\sqrt{-1}\partial\phi_1\wedge\bar{\partial}\phi_1\wedge\chi_k^{m+n-1}
    \end{align*}
    for some constant $C$. On the other hand, by Lemma \ref{lem:dhymevalues}, assumption \eqref{eq:ass}, and
    \begin{align}\label{eq:volume_equivalence}
        \mathrm{Im}(\omega_{k,r}+\sqrt{-1}\chi)^{m+n}=& \sin\theta_\chi(\omega_{k,r}) \left(\prod_i \sqrt{\lambda_i^2(k,r)+1}\sqrt{-1}dz^i\wedge d\bar{z}^i\right),
    \end{align}
    we see that the volume form $\mathrm{Im}(\omega_{k,r}+\sqrt{-1}\chi)^{m+n}$ is equivalent to the volume form $\chi_k^{m+n}$ uniformly for $k\gg0$.
    Let us denote the $\chi_k$-mean value of $\phi_1$ by $m$. Then, as in the proof of \cite[Lemma 6.5]{Fine}, we have
    \begin{align*}
        &\int_X(\phi_1-m)^2\mathrm{Im}(\omega_{k,r}+\sqrt{-1}\chi)^{m+n}\\
        =&\int_X\phi_1^2\mathrm{Im}(\omega_{k,r}+\sqrt{-1}\chi)^{m+n}+m^2\int_X\mathrm{Im}(\omega_{k,r}+\sqrt{-1}\chi)^{m+n}\\
        \ge&\int_X\phi_1^2\, \mathrm{Im}(\omega_{k,r}+\sqrt{-1}\chi)^{m+n},
    \end{align*}
    where the equality follows from the fact that an eigenfunction $\phi_1$ satisfies $\int_X\phi_1\mathrm{Im}(\omega_{k,r}+\sqrt{-1}\chi)^{m+n}=0$ by Lemma \ref{lem:dhymadjoint}. By combining these estimates, we have
    \begin{align*}
        \tilde{\lambda}_1
        &\ge Ck^{-1}\left(\int_X\sqrt{-1}\partial(\phi_1-m)\wedge\bar{\partial}(\phi_1-m)\wedge\chi_k^{m+n-1}\right)\Big/\left(\int_X(\phi_1-m)^2\chi_k^{m+n}\right)\\
        &\ge C'k^{-2}
    \end{align*}
    for some constants $C$ and $C'$, where we used \cite[Lemma 6.5]{Fine} in the last inequality.
\end{proof}

Hereafter, all Sobolev spaces are considered with respect to $\chi_k$. Denote the set of functions in $L^2_l$ satisfying $\int_X\phi\, \mathrm{Im}(\omega_{k,r}+\sqrt{-1}\chi)^{m+n}=0$ by $L^2_{l,0}$. Note that $L^2_{l.0}$ is a closed codimension-one subspace of $L^2_l$ since the volume form $\mathrm{Im}(\omega_{k,r}+\sqrt{-1}\chi)^{m+n}$ is uniformly equivalent to $\chi_k^{m+n}$ by \eqref{eq:volume_equivalence}. Denote the projection from  $L^2_l$ to $L^2_{l,0}$ by $\tilde{p}$.

\begin{lemma}
    There exist a constant C and an integer $A$, independent of $k$ and $r$, such that for $\phi \in L^2_{l+2}$, an integer $r \ge A$ and any $k \gg 0$ we have
    $$ ||\phi||_{L^2_{l+2}} \le C k^A\left( ||\phi||_{L^2} + ||\tilde{F}_{\omega_{k,r}, \chi}(\phi)||_{L^2_l} \right).$$
\end{lemma}

\begin{proof}
    We follow the same arguments as in the proof of Lemma \ref{elliptic}. As in the proof of Lemma \ref{elliptic}, we denote the linearization of the operator $\cot\theta_\chi(\omega_{k,r}+\sqrt{-1}\partial\bar{\partial}\, \cdot)$ by $\tilde{G}_{k,r}$. By \cite[Lemma 3.3]{JY} and the formula $\mathrm{arctan} \, x=\pi/2-\mathrm{arccot}\, x$, locally we have 
    \begin{equation}\label{eq:dhymlin}
        \tilde{G}_{k,r}(\phi)=(\sin\theta_\chi(\omega_{k,r}))^{-2}(\eta_{k,r})^{i\bar{j}}\partial_i\partial_{\bar{j}}\phi,
    \end{equation}
    where $\eta_{k,r}$ is a Hermitian form defined by 
    \begin{equation}\label{eq:eta}
        (\eta_{k,r})_{i\bar{j}}:=\chi_{i\bar{j}}+(\omega_{k,r})_{i\bar{q}}\chi^{p\bar{q}}(\omega_{k,r})_{p\bar{j}}
    \end{equation}
    and $(\eta_{k,r})^{i\bar{j}}=(\eta_{k,r}^{-1})_{i\bar{j}}$. 
    If we use the same coordinate system $\{ (U_i, \tau_i; w^1,\dots, w^{m+n}) \}^N_{i=1}$ as in the proof of Lemma \ref{elliptic}, by the formula of the inverse of a matrix and the same calculations as those in the proof of Lemma \ref{elliptic}, we see that the eigenvalues  of $\eta^{-1}$ are estimated from below by $Ck^{-1}$ for some constant $C$. Therefore, the arguments in the proof of Lemma \ref{elliptic} work and we obtain the same elliptic estimate for the operator $\tilde{G}_{k,r}$.  We also have
    $$\Big(\tilde{G}_{k,r}(\phi)-\tilde{F}_{\omega_{k,r},\chi}\Big)(\phi)=(m+n)\frac{\sqrt{-1}\partial\cot\theta_\chi(\omega_{k,r})\wedge\bar{\partial}\phi\wedge\mathrm{Im}(\omega_{k,r}+\sqrt{-1}\chi)^{m+n-1}}{\mathrm{Im}(\omega_{k,r}+\sqrt{-1}\chi)^{m+n}}.$$
 Using the same argument as in \cite[Lemma 5.7]{Fine} implies that we have
    \begin{align*}
    \tilde{p}(\cot\theta_\chi(\omega_{k,r}))&= O(k^{-r-1}) &\text{in $C^l(\chi_k)$},\\
    \tilde{p}(\cot\theta_\chi(\omega_{k,r}))&= O(k^{-r-1-n/2}) &\text{in $L^2_l(\chi_k)$}.
    \end{align*}
    Therefore, the same argument as in the proof of Lemma \ref{elliptic} yields the desired estimate for $k\gg0$ if we choose $r\ge A$.
\end{proof}

Combining the lemmas we have proved so far, the same proof as that of Lemma \ref{lem:inverse_estimate} yields the following:

\begin{lemma}
    There exist a constant $C$, which depends on $r$, and an integer A, which is independent of $r$,
    such that for all $\phi \in L^2_{l,0}$ and $k \gg 0$, we have
    $$ ||\tilde{F}_{\omega_{k,r}, \chi}^{-1} (\phi)||_{L^2_{l+2}}
    \le Ck^A ||\phi||_{L^2_l}.$$
\end{lemma}

As in the case of the $J$-equation, let $l$ be an integer that satisfies $l-(m+n)>0$. Then, for a function $\phi\in L^2_{l+2}$, we can define the operator $\cot\theta_\chi(\omega_{k,r}):L^2_{l+2}\to L^2_l$ by $\cot\theta_\chi(\omega_{k,r})(\phi)=\cot\theta_\chi(\omega_{k,r}+\sqrt{-1}\partial\bar{\partial}\phi)$. We also define $\tilde{\mathcal{L}}_{k,r}:=\tilde{p}\, \circ\, \cot\theta_{\chi}(\omega_{k,r})|_{L^2_{l+2,0}}$. Recall that we have
\begin{align*}
    &\Big(D_0\cot\theta_\chi(\omega_{k,r})-\tilde{F}_{\omega_{k,r},\chi}\Big)(\phi)\\
    =&(m+n)\frac{\sqrt{-1}\partial\cot\theta_\chi(\omega_{k,r})\wedge\bar{\partial}\phi\wedge\mathrm{Im}(\omega_{k,r}+\sqrt{-1}\chi)^{m+n-1}}{\mathrm{Im}(\omega_{k,r}+\sqrt{-1}\chi)^{m+n}},
\end{align*}
where $D_0\cot\theta_\chi(\omega_{k,r})$ denotes the linearization of $\cot\theta_\chi(\omega_{k,r})$ at $0$. Note also that there exists a constant $C$ such that the inequality
$$\Vert\tilde{p}(\phi)\Vert_{L^2_l}\le C\Vert\phi\Vert_{L^2_l}$$
always holds for $\phi\in L^2_l$. Indeed, let $m$ be a mean value of $\phi$ with respect to the volume form $\mathrm{Im}(\omega_{k,r}+\sqrt{-1}\chi)^{m+n}$. Then, by uniform equivalence of the volume forms $\mathrm{Im}(\omega_{k,r}+\sqrt{-1}\chi)^{m+n}$ and $\chi_k^{m+n}$, we see that
\begin{align*}
    \int_X(\phi-m)^2\chi_k^{m+n}\le &C\int_X(\phi-m)^2\mathrm{Im}(\omega_{k,r}+\sqrt{-1}\chi)^{m+n} \\
    = &C\int_X \phi\, (\phi-m)\mathrm{Im}(\omega_{k,r}+\sqrt{-1}\chi)^{m+n}\\
    \le&C'\int_X\phi^2 \,\chi_k^{m+n}
\end{align*}
for some constant $C'$.
Since $\tilde{p}\, \circ\, \tilde{F}_{\omega_{k,r},\chi}=\tilde{F}_{\omega_{k,r},\chi}$ by  Lemma \ref{lem:dhymadjoint} and $\tilde{p}$ is linear, we obtain
$$\left\Vert(D_0\tilde{\mathcal{L}}_{k,r}-\tilde{F}_{\omega_{k,r},\chi})(\phi)\right\Vert_{L^2_l}\le ck^{-r-1}\left\Vert\phi\right\Vert_{L^2_{l+2}}.$$
In conclusion, if we choose $r\ge A$, Lemma \ref{lem:inverse} implies that $D_0\tilde{\mathcal{L}}_{k,r}$ is an isomorphism for $k\gg0$ and the operator norm of the inverse $P$ satisfies $\Vert P\Vert_{\mathrm{op}}\le Ck^A$ for some constant $C$.

We can also estimate the radius of the ball on which $\tilde{\mathcal{L}}_{k,r}-D_0\tilde{\mathcal{L}}_{k,r}$ is Lipschitz with constant $1/(2\Vert P\Vert)$ by the same argument as in the case of the $J$-equation. Denote the nonlinear part of $\tilde{\mathcal{L}}_{k,r}$ by $\tilde{\mathcal{N}}_{k,r}$, i.e., $\tilde{\mathcal{N}}_{k,r}=\tilde{\mathcal{L}}_{k,r}-D_0\tilde{\mathcal{L}}_{k,r}$. Then, by the mean value theorem, we have
$$\Vert\tilde{\mathcal{N}}_{k,r}(\phi)-\tilde{\mathcal{N}}_{k,r}(\psi)\Vert_{L^2_l}\le\sup_{f\in[\phi,\psi]}\Vert D_f\tilde{\mathcal{N}}_{k,r}\Vert\Vert\phi-\psi\Vert_{L^2_{l+2}},$$
where $D_f\tilde{\mathcal{N}}_{k,r}$ denotes the derivative of $\tilde{\mathcal{N}}_{k,r}$ at $f$.

\begin{lemma}
    There exists a constant $C$ which is independent of $k$ such that if 
    $ \Vert f \Vert_{L^2_{l+2}} \le \epsilon \, k^{-n/2}$ for a small constant $\epsilon$, 
    we have 
    $$\Vert D_f \tilde{\mathcal{N}}_{k,r} \Vert \le C k^{n/2} ||f||_{L^2_{l+2}}.$$
\end{lemma}

\begin{proof}
    For $\phi\in L^2_{l+2}$, by \eqref{eq:dhymlin}, we have
    \begin{align*}
        \Vert (D_f\tilde{\mathcal{N}}_{k,r})(\phi)\Vert_{L^2_l}=&\Vert(D_f\tilde{\mathcal{L}}_{k,r}-D_0\tilde{\mathcal{L}}_{k,r})(\phi) \Vert_{L^2_l}\\
        \le&C\left\Vert (\partial_p\partial_{\bar{q}}\phi)\,  \left(-(\eta_{k,r,f})^{m\bar{q}}+(\eta_{k,r})^{m\bar{q}}\right)\right\Vert_{L^2_l},
    \end{align*}
    where $\eta_{k,r,f}$ denotes the Hermitian form defined by \eqref{eq:eta} with $\omega_{k,r}$ replaced by $\omega_{k,r,f}$ and we estimate $C^{-1}<\sin\theta_\chi(\omega_{k,r})<C$ for some constant $C$ for $k\gg0$ by Lemma \ref{lem:dhymevalues} in the inequality. As in the proof of Lemma \ref{lem:lipschitz}, by using the Sobolev constant with respect to $\chi_k$, which is independent of $k$ by \cite[Lemma 5.8]{Fine}, we obtain
    \begin{align*}
        \Vert (D_f\tilde{\mathcal{N}}_{k,r})(\phi)\Vert_{L^2_l}\le C\Vert\phi\Vert_{L^2_{l+2}}\Vert-\eta_{k,r,f}^{-1}+\eta_{k,r}^{-1}\Vert_{L^2_l}.
    \end{align*}
    Note that we can calculate as 
    \begin{equation}\label{eq:etadiff}
        \begin{split}
            -\eta_{k,r,f}^{-1}+\eta_{k,r}^{-1}
            =&\eta_{k,r,f}^{-1}(\omega_{k,r,f}\chi^{-1}\omega_{k,r,f}-\omega_{k,r}\chi^{-1}\omega_{k,r})\eta_{k,r}^{-1}\\
            =&\eta_{k,r,f}^{-1}(\sqrt{-1}\partial\bar{\partial}f)\chi^{-1}\omega_{k,r,f}+\eta_{k,r,f}^{-1}\omega_{k,r}\chi^{-1}(\sqrt{-1}\partial\bar{\partial}f)\eta_{k,r}^{-1}
        \end{split}
    \end{equation}
    We claim that $\Vert\eta_{k,r}^{-1}\Vert_{C^{l+2}(\chi_k)}$ and $\Vert\omega_{k,r}\Vert_{C^{l+2}(\chi_k)}$ are uniformly bounded above. If this claim is true, we see that $\Vert\eta_{k,r}^{-1}\Vert_{C^{l+2}(\chi_k)}$ and $\Vert\omega_{k,r}\Vert_{C^{l+2}(\chi_k)}$ are estimated above by $Ck^{n/2}$ for some constant $C$ independent of $k$. Then, if we assume $\Vert f\Vert_{L^2_{l+2}}\le \epsilon\, k^{-n/2}$ for a small constant $\epsilon$, we obtain $\Vert\omega_{k,r,f}\Vert_{L^2_l}\le\Vert\omega_{k,r}\Vert_{L^2_l}+C$ for some constant $C$. Evetually by the same argument as in the proof of Lemma 
    \ref{lem:lipschitz} and \eqref{eq:etadiff}, we can estimate $\Vert\eta_{k,r,f}^{-1}\Vert_{L^2_l}$ from above by $Ck^{n/2}$. Therefore, by using \eqref{eq:etadiff} again, we obtain the desired estimate. Lastly, we confirm the claim above. To calculate, let us use a coordinate corresponding to the splitting \eqref{eq:split}. Then, the form $\chi_k$ and $\chi$ do not have mixed terms. Note also that the Christoffel symbols of $\chi_k$ are bounded uniformly. Therefore, by the definitions of $\omega_{k,r}$, \eqref{eq:dhymapp}, and $\eta_{k,r}$, \eqref{eq:eta}, direct computations show that the claim holds.
\end{proof}
As in the case of the $J$-equation, the estimates we obtained allow us to use Theorem \ref{thm:IFT}. Therefore, we complete the proof of Theorem \ref{thm:dhym_main}.

\subsection{A variant of Theorem \ref{thm:converse}}
In this subsection, we prove a variant of Theorem \ref{thm:converse} for the dHYM equation (Theorem \ref{thm:dhym_converse}).
Similarly to the $J$-equation, the solvability of the supercritical dHYM equation is characterized by the following topological condition:

\begin{theorem}[{\cite{G.Chen, CLT}}]\label{thm:CLT}
    Let $(X,\chi)$ be compact $n$-dimensional K{\"a}hler manifold and $\omega$ a closed real $(1,1)$-form. Suppose the supercritical condition holds (see the first part of Subsection \ref{sec:dhymfirst} for the definition). Then, the following are equivalent:
    \begin{enumerate}[font=\normalfont]
        \item There exists a solution $\omega'\in[\omega]$ of the dHYM$_\chi$ equation.
        \item Let $\omega_{t,0}$ be a test family emanating from $\omega_0\in[\omega]$, namely $\omega_{t,0}$ satisfies the following conditions:
        \begin{enumerate}[font=\normalfont]
            \item $\omega_{0,0}=\omega_0\in[\omega]$.
            \item For any $t_1<t_2$, we have $\omega_{t_1,0}<\omega_{t_2,0}$.
            \item There exists a number $T\ge0$ such that $\omega_{t,0}-\cot(\frac{\theta}{n})\chi>0$ holds for all $t\in[T,\infty)$.
        \end{enumerate}
        Then, for any $p$-dimensional subvariety $W\subset X$, where $p=1,2,\dots ,{n-1}$, and for any $t\in[0,\infty)$, we have
        $$\int_W\left(\mathrm{Re}(\omega_{t,0}+\sqrt{-1}\chi)^p-\cot\theta\, \mathrm{Im}(\omega_{t,0}+\sqrt{-1}\chi)^p\right)>0.$$
    \end{enumerate}
\end{theorem}

Related to the above theorem, by the same arguments as in the proof of Theorem \ref{thm:converse}, we can prove the following:
\begin{theorem}\label{thm:dhym_converse}
    Assume Setup \ref{setup:dhym} with $\theta_\infty\in(0,\pi)+2\pi\mathbb{Z}$. Suppose that for any $p$-dimensional subvariety $W\subset X$, where $p=1,2,\dots,{m+n-1}$, there exists a positive constant $k_W$ such that we have
    $$\int_W\left(\mathrm{Re}(\omega_k+\sqrt{-1}\chi)^p-\cot\theta\, \mathrm{Im}(\omega_k+\sqrt{-1}\chi)^p\right)\ge0$$
    for $k\ge k_W$. Then, on each fiber $X_b$, we have
    $$\int_{W_b}\left(\mathrm{Re}(\omega_b+\sqrt{-1}\chi)^p-\cot\theta_\infty\, \mathrm{Im}(\omega_b+\sqrt{-1}\chi)^p\right)\ge0$$
    for any $p$-dimensional subvariety $W_b\subset X_b$, where $p=1,2,\dots,m-1$. In addition, if the restriction $\omega_b$ is a solution of the dHYM$_{\chi_b}$ equation, then the pair $([\omega_B],[\chi_B])$ is $J$-nef.
\end{theorem}

\begin{proof}
   By \eqref{eq:dhymcexpan}, the same arguments as in the proof of Theorem \ref{thm:converse} imply the first statement. Suppose now that the restriction $\omega_b$ is a solution of the dHYM$_{\chi_b}$ equation. Then, the $k^{-1}$-order term of \eqref{eq:dhymcexpan} becomes
    $$-k^{-1}\frac{n}{m+1}\frac{\int_B\chi_B\wedge\omega_B^{n-1}}{\int_B\omega_B^n}.$$
    Therefore, the same arguments as in the proof of Theorem \ref{thm:converse} imply the second statement.
\end{proof}

\begin{remark}\label{rem:other}
    In this paper, we discussed the case $\theta_\infty\in(0,\pi)+2\pi\mathbb{Z}$. The other cases can be treated by the same manner as follows.
    \begin{itemize}
        \item The case $\theta_\infty\in(\pi,2\pi)+2\pi\mathbb{Z}$: we consider the same operator $\cot\theta_\chi(\omega_k)$ and the same form $\chi_B$ defined in Lemma \ref{lem:dhymkahler}, which is K{\"a}hler by the same proof. The only difference is that $-\mathrm{Im}(\omega_X+\sqrt{-1}\chi)^m_{\mathcal{V}}$ becomes a volume form on each fiber.
        \item The case $\theta_\infty\in(-\pi/2,\pi/2)+2\pi\mathbb{Z}$: we consider the operator $\tan\theta_\chi(\omega_k)$.
        \item The case $\theta_\infty\in(\pi/2,3\pi/2)+2\pi\mathbb{Z}$: we consider the operator $\tan\theta_\chi(\omega_k)$.
    \end{itemize}
    In the last two cases, the K{\"a}hler form $\chi_B$ is defined by 
    $$\chi_B:=\frac{\pi_*\Bigl(\mathrm{Im}(\omega_X+\sqrt{-1}\chi)^{m+1}-\tan\theta_\infty\,\mathrm{Re}(\omega_X+\sqrt{-1}\chi)^{m+1}\Bigr)}{\pi_*\Bigl(\mathrm{Re}(\omega_X+\sqrt{-1}\chi)^m\Bigr)}.$$
\end{remark}

\section{Examples}\label{sec:examples}
In this last section, we provide some examples where our main results are applicable. To apply our main results, we need to have an example in which the existence of solutions of the $J$-equations and the dHYM equations is known. These are, for example, compact Riemann surfaces and the projective space $\mathbb{P}^n$. Since both have Picard number one, for any pair of $(1,1)$ cohomology groups, we have trivial solutions of the $J$-equation and the dHYM equation \eqref{eq:dhymlift}. That is, for a Kähler form $\chi$ and $p\in\mathbb{R}$, we have
\begin{equation*}
    \Lambda_\chi p\chi=pn, \quad \theta_\chi(p\chi)=n\operatorname{arccot}p,
\end{equation*}
where $\theta_\chi$ is defined in \eqref{eq:theta}.
The following lemma guarantees that once we can solve an equation on each fiber, we can find a form whose restriction on each fiber is a solution of the equation:

\begin{lemma}
    Assume setup \ref{setup}. Suppose that on each fiber $X_b$, there exists a solution $\omega'_b\in[\omega_X]$ of the $J_{\chi_b}$-equation. Then, there exists a relatively Kähler form $\omega_X'\in[\omega_X]$ such that $\omega_X'|_{X_b}=\omega_b'$ on each fiber. The same type of result also holds for the dHYM equation in setup \ref{setup:dhym}.
\end{lemma}

\begin{proof}
    From Lemma \ref{lem:properties}, the arguments as in the proof of \cite[Lemma 2.1]{Fine2} follow verbatim for the $J$-equation. Namely, since a solution of the $J$-equation is unique and $\operatorname{ker}F_{\omega_b,\chi_b}=\mathbb{R}$, a smooth function $\phi_b$ on $X_b$ that satisfies $\omega'_b=\omega_X|_{X_b}+\sqrt{-1}\partial\bar{\partial}\phi_b$ is smooth in $b$. Therefore, the function $\phi$ in $X$, defined by $\phi(x):=\phi_{\pi(x)}(x)$, is smooth and $\omega'_X=\omega_X+\sqrt{-1}\partial\bar{\partial}\phi$ is the desired form. The same arguments also follow for the dHYM equation, since a solution of the dHYM equation is unique by \cite[Theorem 1.1]{JY} and Lemma \ref{lem:dhymadjoint} (and its variant for $\tan\theta_\chi(\omega)$) holds.
\end{proof}

Using this lemma, we provide two examples of our main results. 

\begin{corollary}\label{cor:surf}
    \begin{enumerate}
        \item In setup \ref{setup}, assume that $X$ is two-dimensional and $B$ is one-dimensional.
        Then, there exists a solution of the $J_\chi$-equation in $[\omega_X+k\omega_B]$ for $k\gg0$. 
        \item In setup \ref{setup:dhym}, assume that $X$ is two-dimensional and $B$ is one-dimensional. Then, there exists a solution of the dHYM$_\chi$ equation in $[\omega_X+k\omega_B]$ for $k\gg0$.
    \end{enumerate}
\end{corollary}

\begin{corollary}\label{cor:proj}
    \begin{enumerate}
        \item In setup \ref{setup}, assume that $X$ is a projective bundle $\mathbb{P}(E)$ and $B$ is the projective space $\mathbb{P}^{n}$.
        Then, there exists a solution of the $J_\chi$-equation in $[\omega_X+k\omega_B]$ for $k\gg0$. 
        \item In setup \ref{setup:dhym}, assume that $X$ is a projective bundle $\mathbb{P}(E)$ and $B$ is the projective space $\mathbb{P}^{n}$. Then, there exists a solution of the dHYM$_\chi$ equation in $[\omega_X+k\omega_B]$ for $k\gg0$.
    \end{enumerate}
\end{corollary}

Remark that the first item of Corollary \ref{cor:surf} has already been proved by \cite[equation (14)]{SD}. Also, the first item of Corollary \ref{cor:proj} in the case 
\begin{equation}
    \begin{aligned}
        &E=(\mathcal{O}_{\mathbb{P}^n}\oplus \mathcal{O}_{\mathbb{P}^n}(-1)^{\oplus m}), \quad [\omega_X]=\pi^*c_1(\mathcal{O}_{\mathbb{P}^n}(1))+b[D_\infty],\\ 
        &[\chi]=\pi^*c_1(\mathcal{O}_{\mathbb{P}^n}(1))+b'[D_\infty], \quad[\omega_B]=c_1(\mathcal{O}_{\mathbb{P}^n}(1)), \quad b>0,b'>0,
    \end{aligned}
\end{equation} 
has already been proved by \cite[Theorem 1.7]{FL} (see also \cite{DMS}). Here, the cohomology class $[D_\infty]$ is the one associated to the divisor $D_\infty:=\mathbb{P}(O_{\mathbb{P}^n}(-1)^{\oplus m})$.
The second item of Corollary \ref{cor:proj} in the case of the blowup of $\mathbb{P}^{n+1}$ at a point, that is,
\begin{equation}
    \begin{aligned}
        &E=(\mathcal{O}_{\mathbb{P}^n}\oplus \mathcal{O}_{\mathbb{P}^n}(-1)), \quad [\omega_X]=
        s\pi^*c_1(\mathcal{O}_{\mathbb{P}^n}(1))+t[D_\infty]=(s+t)[H]-s[E],\\ 
        &[\chi]=\pi^*c_1(\mathcal{O}_{\mathbb{P}^n}(1))+b'[D_\infty]=(1+b')[H]-[E],\quad [\omega_B]=c_1(\mathcal{O}(1)), \quad b'>0,
    \end{aligned}
\end{equation} 
has already been proved by \cite{JS}. Here, the cohomology class $[H]$ is a pullback of the hyperplane bundle on $\mathbb{P}^{n+1}$ and the cohomology class $[E]$ is the one associated with the exceptional divisor.  We can see that the result \cite{JS} implies Corollary \ref{cor:proj} (2) in this case since the sufficient condition in \cite[Lemma 1]{JS} holds for $k\gg0$.


\begin{thebibliography}{99}
 \bibitem[Che21]{G.Chen}
  G. Chen,
  \textit{The J-equation and the supercritical deformed Hermitian-Yang-Mills equation,}
  Invent. Math. \textbf{225} (2021), no. 2, 529--602.
 
 \bibitem[Che00]{Chen}
   X. X. Chen,
   \textit{On the lower bound of the Mabuchi K-energy and its application,}
   Internat. Math. Res. Notices 2000, no. 12, 607--623.

 \bibitem[CC21]{CC}
  X. X. Chen and J. Cheng,
  \textit{On the constant scalar curvature Kähler metrics (II)—Existence results,} 
  J. Amer. Math. Soc. \textbf{34} (2021), no. 4, 937--1009.
 \bibitem[CJY20]{CJY}
  T. C. Collins, A. Jacob, and S.-T. Yau, 
  \textit{$(1,1)$ forms with specified Lagrangian phase: a priori estimates and algebraic obstructions,}
  Camb. J. Math. \textbf{8} (2020), no. 2, 407--452.
 \bibitem[CLT24]{CLT}
  J. Chu, M.-C. Lee, and R. Takahashi, 
  \textit{A Nakai-Moishezon type criterion for supercritical deformed Hermitian-Yang-Mills equation,} 
  J. Differential Geom. \textbf{126} (2024), no. 2, 583--632.
 \bibitem[CS17]{CS}
  T. C. Collins and G. Sz{\'e}kelyhidi, 
  \textit{Convergence of the $J$-flow on toric manifolds,} 
  J. Differential Geom. \textbf{107} (2017), no. 1, 47--81.
  
 \bibitem[CXY18]{CXY}
  T. C. Collins, D. Xie, and S.-T. Yau, 
  \textit{The deformed Hermitian-Yang-Mills equation in geometry and physics. Geometry and physics,} 
  Vol. I, 69--90, Oxford Univ. Press, Oxford, 2018.
 \bibitem[DMS23]{DMS}
  V. V. Datar, R. Mete, and J, Song,
  \textit{Minimal slopes and bubbling for complex Hessian equations,}
  preprint, available at arXiv:2312.03370v1.
 
 \bibitem[DP21]{Datar-Pingali}
   V. V. Datar and V. P. Pingali,
   \textit{A numerical criterion for generalised Monge-Ampère equations on projective manifolds,}
   Geom. Funct. Anal. \textbf{31} (2021), no. 4, 767--814.

 \bibitem[DS21a]{Dervan-Sektnan}
  R. Dervan and L. Sektnan,
  \textit{Optimal symplectic connections on holomorphic submersions,}
  Comm. Pure Appl. Math. \textbf{74} (2021), no. 10, 2132--2184.
  
  \bibitem[DS21b]{DS20}
  R. Dervan and L. M. Sektnan, 
  \textit{Uniqueness of optimal symplectic connections,} 
  Forum Math. Sigma \textbf{9} (2021), Paper No. e18, 37 pp.  
 \bibitem[Don02]{Donaldson}
  S. K. Donaldson,
  \textit{Floer Homology Groups in Yang-Mills Theory,}
  With the assistance of M. Furuta and D. Kotschick,
  Cambridge Tracts in Mathematics, 147. Cambridge University Press, Cambridge, 2002. viii+236 pp.
  
 \bibitem[FL13]{FL}
  H. Fang and M. Lai, \textit{Convergence of general inverse $\sigma_k$-flow on Kähler manifolds with Calabi ansatz,} 
  Trans. Amer. Math. Soc. \textbf{365} (2013), no. 12, 6543--6567.
 \bibitem[Fin04]{Fine}
  J. Fine,
  \textit{Constant scalar curvature K{\"a}hler metrics on fibred complex surfaces,}
  J. Differential Geom. \textbf{68} (2004), no. 3, 397--432.
  
 \bibitem[Fin07]{Fine2}
  J. Fine, \textit{Fibrations with constant scalar curvature Kähler metrics and the CM-line bundle,} 
  Math. Res. Lett. \textbf{14} (2007), no. 2, 239--247.
 \bibitem[GPT23]{GPT}
  B. Guo, D. H. Phong, and F. Tong, 
  \textit{On $L^\infty$ estimates for complex Monge-Ampère equations,} Ann. of Math. (2) \textbf{198} (2023), no. 1, 393--418.
 
 \bibitem[Has19]{Hashimoto2018}
  Y. Hashimoto,
  \textit{Existence of twisted constant scalar curvature K{\"a}hler metrics with a large twist,}
  Math. Z. \textbf{292} (2019), no. 3-4, 791--803.
  
 \bibitem[JS22]{JS}
  A. Jacob and N. Sheu, 
  \textit{The deformed Hermitian-Yang-Mills equation on the blowup of $\Bbb P^n$,} 
  Asian J. Math. \textbf{26} (2022), no. 6, 847--864.
 
 \bibitem[JY17]{JY}
  A. Jacob and S.-T. Yau, 
  \textit{A special Lagrangian type equation for holomorphic line bundles,} Math. Ann. \textbf{369} (2017), no. 1-2, 869--898.
 
 \bibitem[LS15]{LS}
  M. Lejmi and G. Sz{\'e}kelyhidi,
  \textit{The J-flow and stability,}
  Adv. Math. \textbf{274} (2015), 404–431.
 
 \bibitem[LYZ01]{LYZ}
  N. C. Leung, S.-T. Yau, and E. Zaslow, 
  \textit{From special Lagrangian to Hermitian-Yang-Mills via Fourier-Mukai transform,} 
  Winter School on Mirror Symmetry, Vector Bundles and Lagrangian Submanifolds (Cambridge, MA, 1999), 209--225, AMS/IP Stud. Adv. Math., \textbf{23}, Amer. Math. Soc., Providence, RI, 2001.
 
 \bibitem[MMMS00]{MMMS}
  M. Mariño, R. Minasian, G. Moore, and A. Strominger, 
  \textit{Nonlinear instantons from supersymmetric $p$-branes,} 
  J. High Energy Phys. 2000, no. 1, Paper 5, 32 pp.
  
 \bibitem[SD20]{SD}
  Z. Sjöström Dyrefelt, \textit{Optimal lower bounds for Donaldson's J-functional,} Adv. Math. \textbf{374} (2020), 107271, 37 pp.
 
 \bibitem[Son20]{Song}
  J. Song,
  \textit{Nakai-Moishezon criterions for complex Hessian equations,}
  preprint, available at arXiv:2012.07956:v1.

 \bibitem[SW08]{SW}
  J. Song and B. Weinkove, 
  \textit{On the convergence and singularities of the $J$-flow with applications to the Mabuchi energy,} 
  Comm. Pure Appl. Math. \textbf{61} (2008), no. 2, 210--229. 
 
 \bibitem[SYZ96]{SYZ}
  A. Strominger, S.-T. Yau, and E. Zaslow, 
  \textit{Mirror symmetry is $T$-duality,}
  Nuclear Phys. B \textbf{479} (1996), no. 1-2, 243--259.
 
 \bibitem[Sz{\'e}14]{szekelyhidi2014introduction}
  G. Sz{\'e}kelyhidi,
  \textit{An Introduction to Extremal K{\"a}hler Metrics,}
  Graduate Studies in Mathematics, 152. American Mathematical Society, Providence, RI, 2014. xvi+192 pp

 \bibitem[Tay11]{Taylor}
  M. E. Taylor,
  \textit{Partial differential equations I. Basic theory,} 
  Second edition. Applied Mathematical Sciences, 115. Springer, New York, 2011. xxii+654 pp.
\end{thebibliography}
\end{document}